\documentclass[graybox, envcountsame, envcountsect]{svmult}
\usepackage[bottom]{footmisc}
\date{13.6.2013} 
\usepackage{hyperref}
\usepackage{amssymb, amsmath}
\input{ownmacs.sty} 
\newcommand{\mlabel}{\label}

\newcommand\bH{\Bbb{H}} 
 
\newcommand\cH{\mathcal{H}} 
\newcommand\cM{\mathcal{M}} 
 
\newcommand\cD{\mathcal{D}} 
\newcommand\cF{\mathcal{F}} 
\newcommand\cJ{\mathcal{J}} 
\newcommand\cI{\mathcal{I}} 
\newcommand\cK{\mathcal{K}} 
\newcommand\cL{\mathcal{L}} 
\newcommand\cW{\mathcal{W}} 
\newcommand\cQ{\mathcal{Q}}

\newcommand{\U}{\mathop{\rm U{}}\nolimits} 
\newcommand\Skew{\mathop{\rm Skew{}}\nolimits}
\newcommand\Herm{\mathop{\rm Herm{}}\nolimits}
\newcommand\Aherm{\mathop{\rm Aherm{}}\nolimits}
\newcommand\Asym{\mathop{\rm Asym{}}\nolimits}
\newcommand\Sp{\mathop{\rm Sp{}}\nolimits}
\newcommand\Spann{\mathop{\rm span{}}\nolimits}
\newcommand\supp{\mathop{\rm supp{}}\nolimits}

\newcommand{\fS}{{\mathfrak S}}
\newcommand{\shalf}{{\textstyle{\frac{1}{2}}}}
\newcommand{\Part}{\mathop{{\rm Part}}\nolimits}
\newcommand{\indlim}{{\displaystyle \lim_{\longrightarrow}}\ }

\newcommand{\pmat}[1]{\begin{pmatrix} #1 \end{pmatrix}}

\begin{document} 

\title*{Unitary Representations of Unitary Groups} 
\author{Karl-Hermann Neeb}
\institute{Department of Mathematics, 
Friedrich--Alexander University, 
Erlangen-Nuremberg, 
Cauerstrasse 11, 91058 Erlangen, Germany,  
Email: neeb@math.fau.de}%
\maketitle

\abstract{In this paper we review and streamline 
some results of Kirillov, Olshanski and 
Pickrell on unitary representations of the unitary 
group $\U(\cH)$ of a 
real, complex or quaternionic  separable 
Hilbert space and the subgroup $\U_\infty(\cH)$, consisting of those unitary 
operators $g$ for which $g - \1$ is compact. 
The Kirillov--Olshanski theorem on the continuous unitary 
representations of the identity component $\U_\infty(\cH)_0$ 
asserts that they  are direct sums of irreducible ones which can be realized in finite tensor products 
of a suitable complex Hilbert space. This is proved and generalized to 
inseparable spaces. 
These results are carried over to the full 
unitary group by Pickrell's Theorem, asserting that 
the separable unitary representations of 
$\U(\cH)$, for a separable Hilbert space $\cH$, 
are uniquely determined by their restriction to $\U_\infty(\cH)_0$. 
For the $10$ classical infinite rank symmetric 
pairs $(G,K)$ of non-unitary type, such as $(\GL(\cH),\U(\cH))$, 
we also show  that all separable unitary representations are trivial. 
}

\section*{Introduction} 

One of the most drastic differences between the representation theory of finite-dimensional 
Lie groups and infinite-dimensional ones is that an infinite-dimensional Lie group $G$
may carry many different group topologies and any such topology leads to a 
different class of continuous unitary representations. Another perspective on 
the same phenomenon is that the different topologies on $G$ lead to different 
completions, and the passage to a specific completion reduces  
the class of 
representations under consideration. 

In the present paper we survey results and methods of A.~Kirillov, G.~Olshanski and D.~Pickrell 
from the point of view of Banach--Lie groups. 
In the unitary representation theory of finite-dimensional Lie groups, the starting point 
is the representation theory of compact Lie groups and the prototypical compact Lie group 
is the unitary group $\U(n, \C)$ of a complex $n$-dimensional Hilbert space. 
Therefore any systematic representation theory of infinite-dimensional Banach--Lie groups 
should start with unitary groups of Hilbert spaces. 
For an infinite-dimensional Hilbert space $\cH$, there is a large variety of unitary groups. 
First of all, there is the full unitary group $\U(\cH)$, endowed with the norm topology, 
turning it into a simply connected Banach--Lie group with Lie algebra 
$\fu(\cH) = \{ X \in B(\cH) \: X^*= - X\}$. However, 
the much coarser strong operator topology also turns it into another 
topological group
$\U(\cH)_s$. 
The third variant of a unitary group is the subgroup 
$\U_\infty(\cH)$ of all unitary operators $g$ for which 
$g - \1$ is compact. This is a 
Banach--Lie group whose Lie algebra $\fu_\infty(\cH)$ consists of all compact 
operators 
in $\fu(\cH)$. If $\cH$ is separable (which we assume in this introduction) 
and $(e_n)_{n \in \N}$ is an orthonormal basis, 
then we obtain natural embeddings 
$\U(n,\C) \to \U(\cH)$ whose union $\U(\infty,\C) = \bigcup_{n = 1}^\infty \U(n,\C)$ 
carries the structure of a direct limit Lie group (cf.\ \cite{Gl03}). 
Introducing also the Banach--Lie groups 
$\U_p(\cH)$, consisting of unitary operators $g$, for which 
$g - \1$ is of Schatten class~$p \in [1,\infty]$, i.e., $\tr(|U-1|^p) < \infty$, 
we thus obtain an infinite family of 
groups with continuous inclusions 
\[ \U(\infty,\C) \into \U_1(\cH) \into \cdots \into \U_p(\cH) \into \cdots \into 
\U_\infty(\cH) \into\U(\cH) \to \U(\cH)_s. \] 

The representation theory of infinite-dimensional unitary groups began with 
I.~E.~Segal's paper \cite{Se57}, where he studies unitary
representations of the full group $\U({\cal H})$,  
called {\it physical representations}. These are characterized by the
condition that their differential maps finite rank hermitian
projections to positive operators. Segal shows that 
physical representations decompose discretely into irreducible
physical representations which are precisely those 
occurring in the decomposition of finite tensor products $\cH^{\otimes N}$, $N \in\N_0$. 
It is not hard to see that this tensor product decomposes as in classical 
Schur--Weyl theory: 
\begin{equation}
  \label{eq:schur-weyl}
\cH^{\otimes N} \cong \bigoplus_{\lambda \in \Part(N)} \bS_\lambda(\cH) \otimes \cM_\lambda,
\end{equation}
where $\Part(N)$ is the set of all partitions 
$\lambda = (\lambda_1, \ldots, \lambda_n)$ of $N$, 
$\bS_\lambda(\cH)$ is an irreducible unitary representation of $\U(\cH)$ 
(called a {\it Schur representation}), and 
$\cM_\lambda$ is the corresponding irreducible representation of the symmetric group $S_N$, 
hence in particular finite-dimensional  
(cf.\ \cite{BN12} for an extension of Schur--Weyl theory to irreducible 
representations of $C^*$-algebras). 
In particular, 
$\cH^{\otimes N}$ is a finite sum of irreducible representations of~$\U(\cH)$. 

The representation theory of the Banach Lie group  
$\U_\infty({\cal H})$, ${\cal H}$ a separable 
real, complex or quaternionic Hilbert space, was initiated by 
A.\ A.\ Kirillov in \cite{Ki73} (which contains no proofs), and continued by 
G.\ I.\ Olshanski (\cite[Thm.~1.11]{Ol78}). They showed that 
all continuous representations of $\U_\infty(\cH)$ are direct sums of irreducible 
representations and that, for $\K = \C$, all the irreducible 
representations are of the form $\bS_\lambda(\cH) \otimes \bS_\mu(\oline\cH)$, where 
$\oline\cH$ is the space $\cH$, endowed with the opposite complex structure. 
They also obtained generalizes for the corresponding groups over real and quaternionic 
Hilbert spaces. It follows in particular that all irreducible
representations $(\pi, \cH_\pi)$ of the Banach--Lie group $\U_\infty(\cH)$ are 
{\it bounded} in the sense that $\pi \: \U_\infty(\cH) \to \U(\cH_\pi)$ is 
norm continuous, resp., a morphism of Banach--Lie groups. 
The classification of the bounded unitary representations 
of the Banach--Lie group $\U_p(\cH)$ remains the same for ${1 < p < \infty}$, 
but, for $p = 1$, factor representations of type II and III exist 
(see \cite{Boy80} for $p = 2$, and \cite{Ne98} for the general case). 
Dropping the boundedness assumptions even leads to a non-type I 
representation theory for $\U_p(\cH)$, $p < \infty$ 
(cf.\ \cite[Thm.~5.5]{Boy80}). We also refer to \cite{Boy93} for an 
approach to Kirillov's classification based on the classification 
of factor representations of $\U(\infty,\C)$ from \cite{SV75}. 

These results clearly show that the group $\U_\infty(\cH)$ is singled out among all its relatives 
by the fact that its unitary representation theory is well-behaved. 
If $\cH$ is separable, then $\U_\infty(\cH)$ is separable, 
so that its cyclic representations are separable as well. 
Hence there is no need to discuss inseparable representations for this group. 
This is different for the Banach--Lie group 
$\U(\cH)$ which has many inseparable bounded irreducible unitary representations coming 
from irreducible representations of the Calkin algebra $B(\cH)/K(\cH)$. 
It was an amazing insight of D.~Pickrell (\cite{Pi88}) that restricting attention to 
representations on separable spaces tames the representation theory 
of $\U(\cH)$ in the sense that all its separable representations 
are actually 
continuous with respect to the strong operator topology, i.e., continuous 
representations of $\U(\cH)_s$. For analogous results on the 
automatic weak continuity of separable representations 
of $W^*$-algebras see \cite{FF57, Ta60}. 
Since $\U_\infty(\cH)_0$ is dense in $\U(\cH)_s$, it follows that 
$\U_\infty(\cH)_0$ has the same separable representation theory as $\U(\cH)_s$. 
As we shall see below, all these results extend to unitary groups of separable 
real and quaternionic Hilbert spaces.  

Here we won't go deeper into the still not 
completely developed representation theory of groups like $\U_2(\cH)$ which 
also have a wealth of projective unitary representations corresponding to non-trivial central Lie group extensions (\cite{Boy84}, \cite{Ne13}). 
Instead we shall discuss the regular types of 
unitary representation and their characterization. 
For the unitary groups, the natural analogs of the 
finite-dimensional compact groups, a regular setup is obtained 
by considering $\U_\infty(\cH)$ 
or the separable representation of $\U(\cH)$. For direct limit groups, such as 
$\U(\infty,\C)$, 
the same kind of regularity is introduced by Olshanski's concept 
of a tame representation. 
Here a fundamental result is that the tame unitary representation of 
$\U(\infty,\C)_0$ are precisely those extending to 
continuous representations of $\U_\infty(\cH)_0$ (\cite{Ol78}; 
Theorem~\ref{thm:samerep}). 

The natural next step is to take a closer look at unitary representations 
of the Banach analogs of non-compact classical groups; we simply call them 
non-unitary groups. There are $10$ natural  
families of such groups that can be realized by $*$-invariant groups of operators 
with a polar decomposition 
\[ G = K \exp \fp, \ \  \mbox{ where }\ \ 
K = \{ g \in G \: g^* =g^{-1}\} \ \ \mbox{ and } \ \ 
\fp = \{ X \in \g \: X^* =X\}\]  
(see the tables in Section~\ref{sec:5}). 
In particular $K$ is the maximal unitary subgroup of $G$. 
In this context Olshanski calls a continuous 
unitary representation of $G$ admissible if its restriction to $K$ is tame. 
For the cases where the symmetric space $G/K$ is of finite rank, 
Olshanski classifies in \cite{Ol78, Ol84} 
the irreducible admissible representations and shows that 
they form a type I representation theory (see also \cite{Ol89}). The voluminous paper \cite{Ol90} deals with the 
case where $G/K$ is of infinite rank. It contains a precise conjecture about the classification 
of the irreducible representations and the observation that, in general, there are admissible 
factor representations not of type~I. We refer to \cite{MN13} for  recent 
results related to Olshanski's conjecture and to \cite{Ne12} for the classification 
of the semibounded projective unitary representations of hermitian 
Banach--Lie groups. Both continue Olshanski's program in the context 
of Banach--Lie groups of operators. 

In \cite[Prop.~7.1]{Pi90}, Pickrell shows for the $10$ classical 
types of symmetric pairs $(G,K)$ of non-unitary type that, 
for $q > 2$, all separable projective unitary representations are trivial 
for the restricted groups $G_{(q)} = K \exp(\fp_{(q)})$ with Lie algebra 
\[ \g_{(q)} = \fk \oplus \fp_{(q)} \quad \mbox{ and }  \quad 
\fp_{(q)} := \fp \cap B_q(\cH), \] 
where $B_q(\cH) \trile B(\cH)$ is the $q$th Schatten ideal. 
This complements the observation 
that admissible representations often 
extend to the restricted groups $G_{(2)}$ (\cite{Ol90}). 
From these results we learn that, for $q > 2$, the groups $G_{(q)}$ are too big 
to have 
non-trivial separable unitary representations and that 
the groups $G_{(2)}$ have just the 
right size for a rich non-trivial separable representation theory. 
An important consequence is that $G$ itself 
has no non-trivial separable unitary representation. 
This applies in particular to the group 
$\GL_\K(\cH)$ of $\K$-linear isomorphisms of a $\K$-Hilbert space $\cH$ 
and the group $\Sp(\cH)$ of symplectic isomorphism of the symplectic space 
underlying a complex Hilbert space~$\cH$. 

This is naturally extended by the fact 
that, for the $10$ symmetric pairs $(G,K)$ of unitary type and $q > 2$, 
all continuous unitary representations of $G_{(q)}$ extend 
to continuous representations of the full group $G$ (\cite{Pi90}). 
This result has interesting consequences for the representation 
theory of mapping  groups. For a compact spin manifold 
$M$ of odd dimension $d$, there are natural homomorphisms 
of the group $C^\infty(M,K)$, $K$ a compact Lie group, 
into $U(\cH \oplus \cH)_{(d+1)}$, corresponding to the 
symmetric pair $(\U(\cH \oplus \cH), \U(\cH))$ 
(cf.\ \cite{PS86}, \cite{Mi89}, \cite{Pi89}). 
For $d =1$, the rich projective representation theory of 
$U(\cH \oplus \cH)_{(d+1)}$ now leads to the unitary positive 
energy representations of loop groups, but for $d > 1$ the 
(projective) unitary representations 
of $U(\cH \oplus \cH)_{(d+1)}$ extend to the full unitary group $U(\cH \oplus \cH)$, so 
that we do not obtain interesting unitary representations of mapping groups. 
However, there are natural homomorphisms 
$C(M,K)$ into the motion group $\cH \rtimes \OO(\cH)$ 
of a real Hilbert space, and this leads to the interesting class of 
energy representations (\cite{GGV80}, \cite{AH78}).

The content of the paper is as follows. 
In the first two sections we discuss some core ideas and methods 
from the work of Olshanski and Pickrell. 
We start in Section~\ref{sec:1} with the concept of a 
{\it bounded topological group}. These are topological groups $G$, 
for which every identity neighborhood $U$ satisfies $G \subeq U^m$ for some 
$m \in \N$. This boundedness condition permits to show that certain 
subgroups of $G$ have non-zero fixed points in unitary 
representations (cf.\ Proposition~\ref{prop:1.6} for a typical result 
of this kind). We continue in Section~\ref{sec:2} with Olshanski's concept
 of an {\it overgroup}. Starting with a symmetric pair 
$(G,K)$ with Lie algebra $\g = \fk \oplus \fp$, 
the overgroup $K^\sharp$ of $K$ is a Lie group with the Lie algebra $\fk + i \fp$. 
We shall use these overgroups for the pairs 
$(\GL(\cH),\U(\cH))$, where $\cH$ is a  real, complex or quaternionic 
Hilbert space. 

In Section~\ref{sec:3} we describe Olshanski's approach to the classification 
of the unitary representations of $K := \U_\infty(\cH)_0$. 
Here the key idea is that any representation of this group 
is a direct sum of representations $\pi$ generated by the fixed 
space $V$ of the subgroup $K_n$ fixing the first $n$ 
basis vectors. It turns out that this space $V$ carries a 
$*$-representation $(\rho,V)$ 
of the involutive semigroup $C(n,\K)$ of contractions on $\K^n$ 
which determines $\pi$ uniquely by a GNS construction. 
Now the main point is to understand which representations of 
$C(n,\K)$ occur in this process, that they are direct sums of irreducible 
ones and to determine the irreducible representations. 
To achieve this goal, we deviate from Olshanski's approach 
by putting a stronger emphasis on analytic positive definite 
functions (cf.\ Appendices~\ref{app:b}). 
This leads to a considerable simplification of the proof 
avoiding the use of zonal spherical functions and expansions 
with respect to orthogonal polynomials. 
Moreover, our technique is rather close to the setting of holomorphic 
induction developed in \cite{Ne10}. In particular, we use 
Theorem~\ref{thm:extension} which is a slight generalization of 
\cite[Thm.~A.7]{Ne12}. 

In Section~\ref{sec:4} we provide a complete proof of Pickrell's  
Theorem asserting that, for a separable Hilbert space $\cH$, 
the groups $\U(\cH)$ and $\U_\infty(\cH)$ have the same 
continuous separable unitary representations. Here the key 
result is that all continuous separable unitary representations 
of the quotient group $\U(\cH)/\U_\infty(\cH)$ are trivial. 
We show that this result carries over to the real and quaternionic case 
by deriving it from the complex case. 

This provides a 
complete picture of the separable representations of 
$\U(\cH)$ and the subgroup $\U_\infty(\cH)$, but there are many subgroups 
in between. This is naturally complemented by Pickrell's result 
that, for the $10$ symmetric pairs $(G,K)$ of unitary type, 
for $q > 2$, all continuous unitary representations of $G_{(q)}$ extend 
to continuous representations of $G$. 
In Section~\ref{sec:5} we show that, 
for the pairs $(G,K)$ of non-compact type, all separable unitary 
representations of $G$ and $G_{(q)}$, $q > 2$, are trivial. 
This is also stated in \cite{Pi90}, 
but the proof is very sketchy. We use an argument based 
on Howe--Moore theory for the vanishing of matrix coefficients.

\acknowledgement{We thank B.~Kr\"otz for pointing 
out the reference \cite{Ma97} and D.~Pickrell for some notes 
concerning his approach to the proof of \cite[Prop.~7.1]{Pi90}.
We are most greatful to D.~Belti\c t\u a, B.~Janssens and C.~Zellner 
for numerous comments on an earlier version of the manuscript.} 

\subsection*{Notation and terminology}

For the non-negative half line we write $\R_+ = [0,\infty[$. 

In the following $\K$ always denotes $\R$, $\C$ or the skew field $\H$ of quaternions. 
We write $\{1,\cI,\cJ,\cI \cJ\}$ for the canonical basis of $\H$ satisfying 
the relations 
\[ \cI^2 = \cJ^2 = -\1 \quad \mbox{ and } \quad \cI\cJ = - \cJ\cI.\] 
For a real Hilbert space $\cH$, we write 
$\cH_\C$ for its complexification, 
and for a quaternionic Hilbert space $\cH$, we write 
$\cH^\C$ for the underlying complex Hilbert space, obtained 
from the complex structure $\cI \in \H$. For a complex Hilbert space 
we likewise write $\cH^\R$ for the underlying real Hilbert space. 

For the algebra $B(\cH)$ of bounded operators on the $\K$-Hilbert space 
$\cH$, the ideal of compact operators 
is denoted $K(\cH) = B_\infty(\cH)$, and for $1 \leq p < \infty$, we write 
\[ B_p(\cH) := \{ A \in K(\cH) \: \tr((A^*A)^{p/2}) = \tr(|A|^p)< \infty\}\]  
for the {\it Schatten ideals}. In particular, 
$B_2(\cH)$ is the space of {\it Hilbert--Schmidt operators} 
and $B_1(\cH)$ the space of {\it trace class operators}. 
Endowed with the operator norm, the groups 
$\GL(\cH)$ and $\U(\cH)$ are Lie groups with the respective Lie algebras
\[ \gl(\cH) = B(\cH) \quad \mbox{ and } \quad 
\fu(\cH) := \{ X \in \gl(\cH) \: X^* = - X\}.\] 
For $1 \leq p \leq \infty$, we obtain Lie groups 
\[ \GL_p(\cH) := \GL(\cH) \cap (\1 + B_p(\cH)) 
\quad \mbox{ and } \quad 
\U_p(\cH) := \U(\cH) \cap \GL_p(\cH) \] 
with the Lie algebras 
\[ \gl_p(\cH) := B_p(\cH) \quad \mbox{ and } \quad 
\fu_p(\cH) := \fu(\cH) \cap \gl_p(\cH). \] 

To emphasize the base field $\K$, we sometimes write 
$\U_\K(\cH)$ for the group $\U(\cH)$ of $\K$-linear isometries of $\cH$. 
We also write $\OO(\cH) = \U_\R(\cH)$. 

If $G$ is a group acting on a set $X$, then we write $X^G$ for the subset of 
$G$-fixed points.

\section{Bounded groups} 
\mlabel{sec:1}

In this section we discuss one of Olshanski's key concepts for the approach to Kirillov's Theorem 
on the classification of the representations of $\U_\infty(\cH)_0$ 
for a separable Hilbert space discussed in Section~\ref{sec:3}. 
As we shall see below (Lemma~\ref{lem:p.1}), this method 
also lies at the heart of Pickrell's Theorem on the separable representations 
of $\U(\cH)$. 

\begin{definition} {\rm We call a 
topological group $G$ {\it bounded} if, for every 
identity neighborhood $U \subeq G$, there exists an $m \in \N$ with 
$G \subeq U^m$. }
\end{definition}

Note that every locally connected bounded topological group is connected. 
The group $\Q/\Z$ is bounded but not connected. 

\begin{lemma} \mlabel{lem:1.1} 
If, for a  Banach--Lie group $G$, there exists a $c > 0$ with 
\[ G = \exp \{ x \in \g \: \|x\| \leq c \} \tag{Ol},\] 
then $G$ is bounded. 
\end{lemma}

\begin{proof} Let $U$ be an identity neighborhood of $G$. 
Since the exponential function $\exp_G \: \g \to G$ 
is continuous, there exists an $r > 0$ with 
$\exp x \in U$ for $\|x\| < r$. Pick $m \in \N$ such that 
$m r > c$. For $g = \exp x$ with $\|x\| \leq c$ we then have 
$\exp\frac{x}{m} \in U$, and therefore $g \in U^m$. 
\smartqed\qed\end{proof}

\begin{proposition} \mlabel{prop:1.3} The following groups satisfy 
(Ol), hence are bounded: 
  \begin{description}
  \item[\rm(i)] The full unitary group $\U(\cH)$ of an infinite-dimensional 
complex or quaternionic Hilbert space. 
  \item[\rm(ii)] The unitary group $\U(\cM)$ of a von Neumann algebra $\cM$.  
  \item[\rm(iii)] The identity component $\U_{\infty}(\cH)_0$ of 
$\U_\infty(\cH)$ for a $\K$-Hilbert 
space~$\cH$.\begin{footnote}{Actually this group is connected for 
$\K = \C,\H$ (\cite[Cor.~II.15]{Ne02}).}\end{footnote}

  \end{description}
\end{proposition}

\begin{proof} (i) {\bf Case $\K =\C$:} Let $g \in \U(\cH)$ and 
let $P$ denote the spectral measure on the unit circle $\T \subeq \C$ with 
$g = \int_\T z \, dP(z)$. We consider the measurable 
function $L \: \T \to ]-\pi, \pi]i$ which is the inverse of the 
function $]-\pi, \pi]i \to \T, z \mapsto e^z$. 
Then 
\begin{equation}\label{eq:xop}
X := \int_\T L(z)\, dP(z) 
\end{equation}
 is a skew-hermitian operator with 
$\|X\| \leq \pi$ and $e^X = g$ (cf.\  \cite[Thm.~12.37]{Ru73}). 


{\bf Case $\K =\H$:} We consider the quaternionic Hilbert space as a 
complex Hilbert space $\cH^\C$, endowed with an anticonjugation 
(=antilinear complex structure) $\cJ$. Then 
\[ \U_\H(\cH) = \{ g \in \U(\cH^\C) \: \cJ g\cJ^{-1} = g \}.\] 
An element $g \in \U(\cH^\C)$ is $\H$-linear if and only if 
the relation $\cJ P(E) \cJ^{-1} = P(\oline E)$ holds for the corresponding 
spectral measure $P$ on $\T$. 

Let $\cH_0 := \ker(g + \1) = P(\{-1\})\cH$ denote the $(-1)$-eigenspace of $g$ 
and $\cH_1 := \cH_0^\bot$. 
If $X$ is defined by \eqref{eq:xop} and $X_1 := X\res_{\cH_1}$, then 
\begin{align*}
 \cJ X_1 \cJ^{-1}
&= \int_{\T \setminus \{-1\}}-L(z)\, dP(\oline z)
= \int_{\T \setminus \{-1\}} L(z)\, dP(z) = X_1. 
\end{align*}
Then $X := \pi \cJ\res_{\cH_0} \oplus X_1$ on $\cH = \cH_0 \oplus \cH_1$ is 
an element $X \in \fu_\H(\cH)$ with $e^X = g$ 
and $\|X\| \leq \pi$. Therefore (Ol) is satisfied. 

(ii) If $g \in \U(\cM)$, then $P(E) \in \cM$ for every measurable subset 
$E \subeq \T$, and therefore $X \in \cM$. Now (ii) follows as (i) for 
$\K = \C$. 

(iii) {\bf Case $\K = \C, \H$:} The operator $X$ from \eqref{eq:xop} is compact 
if $g - \1$ is compact. Hence the group 
$\U_{\infty,\K}(\cH)$ is connected and we can argue as in (i). 

{\bf Case $\K = \R$:} We consider $\U_{\infty,\R}(\cH) = \OO_\infty(\cH)$ as a subgroup 
of $\U_\infty(\cH_\C)$. Let $\sigma$ denote the antilinear 
isometry on $\cH_\C$ whose fixed 
point set is $\cH$. Then, for $g \in \U(\cH_\C)$, the relation 
$\sigma g \sigma = g$ is equivalent to $g \in \OO(\cH)$. This is equivalent to 
the relation $\sigma P(E) \sigma = P(\oline E)$ for the corresponding 
spectral measure on $\T$. 

Next we recall from \cite[Cor.~II.15]{Ne02} 
(see also \cite{dlH72}) that the group $\OO_\infty(\cH)$ 
has two connected
components. An element $g \in \OO_\infty(\cH)$ for which $g -\1$ is of trace class  
is contained in the identity component if and only if 
$\det(g) = 1$. From the normal form of orthogonal compact operators 
that follows from the spectral measure on $\cH_\C$, 
it follows that $\det(g) = (-1)^{\dim \cH^{-g}}$. Therefore 
the identity component of $\OO_\infty(\cH)$ consists of those elements 
$g$ for which the $(-1)$-eigenspace $\cH_0 = \cH^{-g}$ is of even dimension. 
Let $J \in \fo(\cH_0)$ be an orthogonal complex structure. 
On $\cH_1 := \cH_0^\bot$ the operator $X_1 := X\res_{\cH_1}$ 
satisfies 
\begin{align*}
 \sigma X_1 \sigma 
&= \int_{\T \setminus \{-1\}} -L(z)\, dP(\oline z)
= \int_{\T \setminus \{-1\}} L(z)\, dP(z) = X_1,
\end{align*}
so that $X := \pi J \oplus X_1 \in \fo_\infty(\cH)$ satisfies $\|X\| \leq \pi$ and 
$e^X = g$. 
\smartqed\qed\end{proof}

\begin{example} \mlabel{ex:1.2} 
(a) In view of Proposition~\ref{prop:1.3}, it is remarkable that 
the full orthogonal group $\OO(\cH)$ of a real 
Hilbert space $\cH$ does not satisfy (Ol). Actually its exponential function 
is not surjective (\cite{PW52}). In fact, if $g = e^X$ for $X \in \fo(\cH)$, 
then $X$ commutes with $g$, hence preserves the $(-1)$-eigenspace 
$\cH_0 := \ker(g +\1)$. Therefore $J := e^{X/2}$ defines a complex 
structure on $\cH_0$, showing that $\cH_0$ is either infinite-dimensional 
or of even dimension. Therefore no element $g \in \OO(\cH)$ for which 
$\dim \cH_0$ is odd is contained in the image of the exponential function. 

(b) We shall need later that $\OO(\cH)$ is connected. This follows from Kuiper's 
Theorem (\cite{Ku65}), but one can give a more direct argument based on the preceding 
discussion. It only remains to show 
that elements $g \in \OO(\cH)$ for which the space 
$\cH_0 := \ker(g + \1)$ is of finite odd dimension are contained in the identity component. 
We write $g = g_{-1} \oplus g_1$ with $g_{-1} = g\res_{\cH_0}$ and $g_1 := g\res_{\cH_0^\bot}$. 
Then $g_1$ lies on a one-parameter group of $\OO(\cH_0^\bot)$, so that $g$ is connected 
by a continuous arc to $g' := -\1_{\cH_0} \oplus \1$. This element is connected to 
$g'' := -\1_{\cH_0} \oplus - \1_{\cH_0^\bot} = -\1_\cH$, and this in turn to $\1_\cH$. 
Therefore $\OO(\cH)$ is connected. 
\end{example}

\begin{lemma} {\rm(Olshanski Lemma)} 
  \mlabel{lem:1.4} 
For a group $G$, a subset $U \subeq G$ and $m \in \N$ with 
$G \subeq U^m$, we put 
\[ \eta := \sqrt{1 -  \frac{1}{4(m + 1)^2}} \in ]0,1[.\] 
If $(\pi, \cH)$ is a unitary representation with $\cH^G = \{0\}$, 
then, for any non-zero $\xi \in {\cal H}$, there exists $u\in U$ with 
\[  \shalf \| \xi + \pi(u)\xi\| < \eta \|\xi\|. \] 
\end{lemma}

\begin{proof} (\cite[Lemma 1.3]{Ol78}) 
If $\|\xi - \pi(u)\xi\| \leq \lambda \|\xi\|$ holds for all $u \in
U$, then the triangle inequality implies 
\[ \|\xi - \pi(g)\xi\| \leq m \lambda \|\xi\| \quad \mbox{ for } \quad g \in U^m = G.\] 
For $\lambda < {1\over m}$ this implies that the closed convex hull of
the orbit $\pi(G)\xi$ does not contain $0$, hence contains a
non-zero fixed point by the Bruhat--Tits Theorem (\cite{La99}), applied to the 
isometric action of $G$ on $\cH$. This violates our
assumption $\cH^G = \{0\}$. We conclude that there exists a $u \in U$ with 
$\|\xi - \pi(u)\xi\| > {1\over m + 1} \|\xi\|$. 
Thus 
\[ 2\|\xi\|^2 - 2 \Re \la \xi, \pi(u)\xi \ra = \|\xi - \pi(u)\xi\|^2 > {\|\xi\|^2 \over (m
+ 1)^2}, \]
which in turn gives
\[ \|\xi + \pi(u)\xi\|^2 = 2\|\xi\|^2 + 2 \Re \la \xi, \pi(u)\xi \ra 
< \big(4 -  \frac{1}{(m + 1)^2} \big)\|\xi\|^2 = \eta^2 \|\xi\|^2.\qquad\smartqed\qed\]
\end{proof}

The following proposition is an abstraction of the proof of \cite[Lemma~1.4]{Ol78}. 
It will be used in two situations below, to prove Kirillov's Lemma~\ref{lem:p.1x}  and 
in Pickrell's Lemma~\ref{lem:p.1}. 

\begin{proposition} \mlabel{prop:1.6} Let $G$ be a bounded topological group 
and $(G_n)_{n \in \N}$ be a sequence of subgroups of $G$. If there exists a basis 
of $\1$-neighborhoods $U \subeq G$ such that either 
\begin{description}
\item[\rm(a)] 
\begin{description}
\item[\,\rm(U1)] $(\exists m \in \N) (\forall n) \ G_n \subeq (G_n \cap U)^m$, and  
\item[\rm(U2)] $(\forall N \in \N) (G_N \cap U) \cdots (G_1 \cap U) \subeq U$,  
\end{description}
or 
\item[\rm(b)] \begin{description}
\item[\,\rm(V1)] $(\exists m \in \N) (\forall n) \ G_n \subeq (G_n \cap U)^m$, and  
\item[\rm(V2)] there exists an increasing sequence of subgroups 
$(G(n))_{n \in \N}$ such that 
\begin{description}
\item[\rm(1)] $(\forall n \in \N)$ the union of 
$G(m)_n := G(m) \cap G_n$, $m \in \N$, is dense in $G_n$. 
\item[\rm(2)] $(G(k_1)_1 \cap U)(G(k_2)_{k_1} \cap U) 
\cdots (G(k_N)_{k_{N-1}} \cap U) \subeq U$ 
for \break $1 < k_1 < \ldots < k_N$. 
\end{description}
\end{description}
\end{description}
Then there exists an $n \in \N$ with $\cH^{G_n} \not= \{0\}$. 
\end{proposition}

\begin{proof} (a) 
We argue by contradiction and assume that $\cH^{G_n} = \{0\}$ for every $n$. 
Let $\xi \in \cH$ be  non-zero 
and $U \subeq G$ be an identity neighborhood with 
$\|\pi(g)\xi - \xi\| < \shalf \|\xi\|$ for $g \in U$ 
such that (U1/2) are satisfied. Let $\eta$ be as in Lemma~\ref{lem:1.4}. 

Since $G_1$ has no non-zero fixed vector, 
there exists an element $u_1 \in U \cap G_1$ with 
\[  \| \shalf(\xi + \pi(u_1)\xi)\| \leq \eta\|\xi\|. \] 
Then $\xi_1 := \shalf(\xi + \pi(u_1)\xi)$ satisfies 
$\|\xi_1 - \xi\| < \shalf \|\xi\|$, so that 
$\xi_1 \not=0$. Iterating this procedure, we obtain a sequence 
of vectors $(\xi_n)_{n \in \N}$ and elements 
$u_n \in U \cap G_n$ 
with $\xi_{n+1} := \shalf(\xi_n + \pi(u_{n+1})\xi_n)$ 
and $\|\xi_{n+1}\| \leq \eta \|\xi_n\|$. 

We consider the probability measures 
$\mu_n := \shalf(\delta_\1 + \delta_{u_n})$ on $G$ 
and observe that (U2) implies 
$\supp(\mu_n * \cdots * \mu_1) \subeq U$ for every $n \in \N$. 
By construction we have $\pi(\mu_n * \cdots * \mu_1)\xi = \xi_n$, so that 
$\|\xi_n - \xi\| < \shalf \|\xi\|$. On the other hand, 
\[ \|\pi(\mu_n * \cdots * \mu_1)\xi\| 
= \|\xi_n\| \leq \eta^{n} \|\xi\| \to 0,\] 
and this is a contradiction. 

(b) Again, we argue by contradiction and assume that no subgroup 
$G_n$ has a non-zero fixed vector. 
Let $\xi \in \cH$ be a non-zero vector 
and $U$ be an identity neighborhood with 
$\|\pi(g)\xi - \xi\| < \shalf \|\xi\|$ for $g \in U$ such 
that (V1) is satisfied. 

Since $G_1$ has no non-zero fixed point in $\cH$ and 
 $\bigcup_{n=1}^\infty G(n)_1$ is dense in $G_1$, 
there exists a $k_1 \in \N$ and some $u_1 \in U \cap G(k_1)_1$ with 
$\| \shalf(\xi + \pi(u_1)\xi)\| < \eta\|\xi\|$ 
(Lemma~\ref{lem:1.4}). 
For $\xi_1 := \shalf(\xi + \pi(u_1)\xi)$ our construction 
then implies that $\|\xi_1 - \xi\| < \shalf \|\xi\|$, so that 
$\xi_1 \not=0$. Any  $u \in U \cap G_{k_1}$ commutes with $u_1$, so that 
we further obtain 
\begin{align*}
 \|\pi(u)\xi_1 - \xi_1\| 
&= \shalf \|\pi(u)\xi - \xi + \pi(u) \pi(u_1) \xi - \pi(u_1)\xi\|\\
&< \shalf (\shalf \|\xi\| + \|\pi(u_1)  \pi(u) \xi - \pi(u_1)\xi\|)\\
&= \shalf (\shalf \|\xi\| + \|\pi(u) \xi - \xi\|)
<  \shalf (\shalf \|\xi\| + \shalf \|\xi\|) = \shalf \|\xi\|.
\end{align*}
Iterating this procedure, we obtain a strictly increasing sequence 
$(k_n)$ of natural numbers, a sequence 
$(\xi_n)$ in $\cH$ and $u_n \in G(k_n)_{k_{n-1}} \cap U$ 
with 
\[ \xi_{n+1} := \shalf(\xi_n + \pi(u_{n+1})\xi_n) \quad \mbox{ and } \quad 
\|\xi_{n+1} \| < \eta \|\xi_n\|.\] 

We consider the probability measures 
$\mu_n := \shalf(\delta_\1 + \delta_{u_n})$ on $G$. 
Condition (V2)(2) implies that 
\[ \supp(\mu_n * \cdots * \mu_1)
\subeq \{ \1,u_n\} \cdots \{ \1, u_1\} \subeq U\quad \mbox{ for every } \quad 
n \in \N.\] 
By construction $\pi(\mu_n * \cdots * \mu_1)\xi = \xi_n$, so that 
$\|\xi_n - \xi\| < \shalf \|\xi\|$. On the other hand, 
\[ \|\pi(\mu_n * \cdots * \mu_1)\xi\| 
= \|\xi_n\| \leq \eta^n \|\xi\| \to 0,\] 
and this is a contradiction. 
\smartqed\qed\end{proof}

\section{Duality and overgroups} \mlabel{sec:2} 

Apart from the fixed point results related to bounded topological 
groups discussed in the preceding section, another central concept 
in Olshanski's approach are ``overgroups''. 
They are closely related to the duality of symmetric spaces.

\begin{definition} {\rm A {\it symmetric Lie group} is a triple $(G,K,\tau)$, where 
$\tau$ is an involutive automorphism of the Banach--Lie group 
$G$ and $K$ is an open subgroup 
of the Lie subgroup $G^\tau$ of $\tau$-fixed points in~$G$. 
We write $\g = \fk \oplus \fp = \g^\tau \oplus \g^{-\tau}$ 
for the eigenspace decomposition of $\g$ with respect to $\tau$ and  
call $\g^c := \fk \oplus i \fp \subeq \g_\C$ the {\it dual symmetric 
Lie algebra}.} 
\end{definition}

\begin{definition} {\rm Suppose that $(G,K,\tau)$ is a symmetric Lie group 
and $G^c$ a simply connected Lie group 
with Lie algebra $\g^c = \fk + i \fp$. 
Then $X + i Y \mapsto X - i Y$ ($X \in \fk$, $Y \in \fp$), integrates to an involution 
$\tilde\tau^c$ of $G^c$. Let $q_K \: \tilde K_0 \to K_0$ denote the universal 
covering of the identity component $K_0$ of $K$ 
and $\tilde\iota_K \: \tilde K_0 \to G^c$ 
the homomorphism integrating the inclusion 
$\fk \into \g^c$. The group $\tilde\iota_K(\ker q_K)$ 
acts trivially on $\fg_\C$, hence is central in $G^c$. If it is discrete, 
then we call 
\[ (K_0)^\sharp := G^c/\tilde\iota_K(\ker q_K) \] 
the {\it overgroup of $K_0$}. 
In this case $\tilde\iota_K$ factors through a covering map 
$\iota_{K_0} \: K_0 \to (K_0)^\sharp$ 
and the involution $\tau^c$ induced by $\tilde\tau^c$ on $(K_0)^\sharp$ 
leads to a symmetric Lie group $((K_0)^\sharp, \iota_{K_0}(K_0), \tau^c)$. 

To extend this construction to the case where $K$ is not connected, 
we first observe that 
$K \subeq G$ acts naturally on the Lie algebra $\g^c$, hence also on 
the corresponding simply connected group  $G^c$. This action preserves  
$\tilde\iota_K(\ker q_K)$, hence induces an action on $(K_0)^\sharp$, so that 
we can form the semidirect product $(K_0)^\sharp \rtimes K$. 
In this group $N := \{ (\iota_{K_0}(k), k^{-1}) \: k \in K_0\}$ is a closed 
normal subgroup and we put 
\[ K^\sharp := ((K_0)^\sharp \rtimes K)/N, \quad 
\iota_K(k) := (\1,k) N \in K^\sharp.\] 

The overgroup $K^\sharp$ has the universal property that if a 
morphism $\alpha \: K \to H$ of Lie groups extends to a Lie group 
with Lie algebra $\g^c = \fk + i \fp$, then 
$\alpha$ factors through $\tilde\iota_K \: K \to K^\sharp$. 
}\end{definition} 

\begin{example} \mlabel{ex:2.3} (a) If $\cH$ is a $\K$-Hilbert space, then the triple 
$(\GL_\K(\cH), \U_\K(\cH), \tau)$ with $\tau(g) = (g^*)^{-1}$ is a symmetric 
Lie group. For its Lie algebra 
\[ \g = \gl_\K(\cH) = \fk \oplus \fp = \fu_\K(\cH) \oplus \Herm_\K(\cH),\] 
the corresponding dual symmetric Lie algebra is 
\[ \g^c = \fk + i \fp = \fu_\K(\cH) \oplus i\Herm_\K(\cH) \subeq \fu(\cH_\C).\] 
More precisely, we have 
\begin{description}
\item[$(\R)$] $\gl_\R(\cH)^c = \fo(\cH) \oplus i \Sym(\cH) \cong \fu(\cH_\C)$ for $\K = \R$. 
\item[$(\C)$] $\gl_\C(\cH)^c = \fu(\cH) \oplus i \Herm(\cH) \cong\fu(\cH)^2$  for $\K = \C$.  
\item[$(\H)$] $\gl_\H(\cH)^c = \fu_\K(\cH) \oplus \cI \Herm_\H(\cH)\cong \fu(\cH^\C)$  for $\K = \H$. 
\end{description}
Here the complex case requires additional explanation. Let $I$ denote the given complex 
structure on $\cH$. Then the maps 
\[ \iota_\pm \: \cH \to \cH_\C, \quad v \mapsto \frac{1}{\sqrt 2}(v \mp i I v) \] 
are isometries to complex subspaces $\cH_\C^\pm$ of $\cH_\C$, where 
$\iota_+$ is complex linear and $\iota_-$ is antilinear. We thus obtain 
\[ \cH_\C = \cH_\C^+ \oplus \cH_\C^- \cong \cH \oplus \oline\cH,\] 
and $\cH_\C^\pm$ are the $\pm i$-eigenspaces of the complex linear extension of 
$I$ to $\cH_\C$. In particular, $\gl(\cH)^c$ preserves 
both subspaces $\cH_\C^\pm$. This leads to the isomorphism 
\[ \gamma \: \gl(\cH)^c \to 
\fu(\cH_\C^+ ) \oplus \fu(\cH_\C^-) \cong 
\fu(\cH) \oplus \fu(\oline\cH), \quad 
\gamma(X + i Y) = (X + IY, X - I Y).\] 
\end{example}

\begin{lemma}
  \mlabel{lem:k-grp} For a $\K$-Hilbert space $\cH$, 
let $(G,K) = (\GL_\K(\cH), \U_\K(\cH)_0)$ and 
$n := \dim \cH$. Then $K$ is connected for 
$\K \not=\R$ and $n= \infty$, and 
\[  (K_0)^\sharp \cong  
\begin{cases}
\tilde \U(n,\C) & \text{ for } \K = \R, n < \infty \\ 
\tilde \U(n,\C)^2/\Gamma & 
 \text{ for } \K = \C, n < \infty,\ \ 
\Gamma := \{(z,z) \: z \in \pi_1(\U(n,\C))\}. \\ 
\tilde \U(2n,\C) & \text{ for } \K = \H, n < \infty,  \\
\U(\cH_\C) & \text{ for } \K = \R \\ 
\U(\cH)\oplus \U(\oline\cH) &  \text{ for } \K = \C \\ 
\U(\cH^\C)& \text{ for } \K = \H.
\end{cases} \]
Here $\cH^\C$ is the complex Hilbert space underlying a quaternionic 
Hilbert space $\cH$. 
For $\K = \R, \H$, the map $\iota_K$ is the canonical inclusion, and 
$\iota_K(k) = (k,k)$ for $\K = \C$. 
\end{lemma}

\begin{proof} First we consider the case where $n < \infty$. 
 Recall that 
 \begin{align*}
&\tilde\U(n,\C) \cong \SU(n,\C) \rtimes \R, \quad 
\OO(n,\R)_0 = \SO(n,\R) \subeq \SU(n,\C), \quad \U(n,\H) \subeq \SU(2n,\C). 
 \end{align*}
For $\K = \R$, this implies that 
$K = \OO(n,\R)_0$ embeds into $\tilde\U(n,\C)$, so that 
$K^\sharp \cong \tilde\U(n,\C)$. For $\K = \H$, we see that 
$K = \U(n,\H)$ embeds into $\tilde\U(2n,\C)$, which leads to 
$\tilde K \cong \tilde\U(2n,\C)$. 

For $\K = \C$, we have the natural inclusion  
$$ i_K \: K = \U(n,\C) \to \U(n,\C) \times \U(n,\C) \cong \U(\C^n) \times 
\U(\oline{\C^n}), \quad g \mapsto (g,g). $$
To determine $K^\sharp$, we note that the image of 
\[ \pi_1(i_K) \: 
\Z \cong \pi_1(\U(n,\C)) \to \Z^2 \cong \pi_1(\U(n,\C)^2), \quad 
m \mapsto (m,m) \]
is $\Gamma$. Therefore $K^\sharp \cong \tilde\U(n,\C)^2/\Gamma$. 

If $n = \infty$, then $K = \U_\K(\cH)$ is a simply connected Lie group with 
Lie algebra $\fk$ (\cite{Ku65}), so that $K^\sharp$ is the 
simply connected Lie group with Lie algebra $\fk^\sharp$ and we 
have a natural morphism $\iota_K \: K \to K^\sharp$ integrating the inclusion 
$\fk \into \fk^\sharp$. 
\smartqed\qed\end{proof}

\section{The unitary representations of $\U_\infty(\cH)_0$} 
\mlabel{sec:3}

In this section we completely describe the representations 
of the Banach--Lie groups $\U_\infty(\cH)_0$ for an infinite-dimensional 
real, complex or quaternionic Hilbert space. In particular, we 
show that all continuous unitary representations are direct sums of 
irreducible ones and classify the irreducible ones 
(Theorem~\ref{thm:class}). 
Our approach is based on Olshanski's treatment in \cite{Ol78}. 
We also take some short cuts that simplify the proof 
and put a stronger emphasis on analytic positive definite functions. 
This has the nice side effect, that we also obtain these 
results for inseparable Hilbert spaces 
(Theorem~\ref{thm:insep}).

\subsection{Tameness as a continuity condition} 
\mlabel{subsec:3.1}

We start with a brief discussion of Olshanski's concept 
of a tame representation that links representations 
of $\U_\infty(\cH)_0$ to representations of the direct limit group~$\U(\infty,\K)$. 

Let $K$ be a group and $(K_j)_{j \in J}$ a non-empty 
family of subgroups satisfying the 
following conditions 
\begin{description}
\item[\rm(S1)] It is a {\it filter basis}, i.e., for 
$j,m \in J$, there exists an $\ell \in J$ with 
$K_\ell \subeq K_j \cap K_m$. 
\item[\rm(S2)] $\bigcap_{j \in J} K_j = \{\1\}$. 
\item[\rm(S3)] For each $g \in K$ and $j \in J$ there exists an $m \in J$ with 
$g K_m g^{-1} \subeq K_j$. 
\end{description}
Then there exists a unique Hausdorff group topology $\tau$ on $K$ for which 
$(K_j)_{j \in J}$ is a basis of $\1$-neighborhoods (\cite[Ch.~4]{Bou98}). We call $\tau$ the 
{\it topology defined by $(K_j)_{j \in J}$}. 

\begin{definition}
  \mlabel{def:2.1} {\rm 
We call a unitary representation $(\pi, \cH)$ of $K$ {\it tame}
if the space 
$$ {\cal H}^\cT 
:= \sum_{j \in J} {\cal H}^{K_j} 
= \bigcup_{j \in J} {\cal H}^{K_j}, $$
is dense in ${\cal H}$. Note that, for $K_j \subeq K_k \cap K_\ell$, we have 
${\cal H}^{K_j}  \supeq {\cal H}^{K_k} + {\cal H}^{K_\ell},$
so that ${\cal H}^\cT$ is a directed union of the 
closed subspaces $\cH^{K_j}$. 
}\end{definition}

\begin{lemma}
  \mlabel{lem:2.2}
A unitary representation of $K$ is tame if and only if
it is continuous with respect to the group topology defined by the
filter basis $(K_j)_{j \in J}$. 
\end{lemma}

\begin{proof}
If $(\pi, {\cal H})$ is a tame representation, then 
${\cal H}^\cT$ obviously consists of continuous vectors for $K$
since, for each $v \in {\cal H}^\cT$, the stabilizer is open. 
Hence the set of continuous vectors is dense, and therefore $\pi$ is
continuous. 

If, conversely, $(\pi, {\cal H})$ is continuous and $v \in {\cal H}$,
then the orbit map $K \to {\cal H}, g \mapsto gv$ is continuous. 
Let $B_\eps$ denote the closed $\eps$-ball in ${\cal H}$. 
Then there exists a $j \in J$ with 
$\pi(K_j)v \subeq v + B_\eps$. Then 
$C := \oline{\conv(\pi(K_j)v)}$ is a closed convex invariant subset of $v
+ B_\eps$, hence contains a $K_j$-fixed point 
by the Bruhat--Tits Theorem (\cite{La99}). 
This proves that 
${\cal H}^{K_j}$ intersects $v + B_\eps$, hence that ${\cal H}^\cT$
is dense in~${\cal H}$. 
\smartqed\qed\end{proof}

\begin{remark} \mlabel{rem:2.3} (a) For a unitary representation 
$(\pi, \cH)$ of a topological group, the subspace 
$\cH^c$ of continuous vectors is closed and invariant. 
The representation $\pi$ is continuous if and only if $\cH^c = \cH$. 

(b) For a unitary 
representation $(\pi, \cH)$ of $K$, 
by Lemma~\ref{lem:2.2}, the space of continuous vectors coincides with 
$\oline{\cH^\cT}$. In particular, it is $K$-invariant. 

(c) If the representation $(\pi, \cH)$ of $K$ is irreducible, 
then it is tame if and only if $\cH^\cT \not=\{0\}$. 

(d) If the representation $(\pi, \cH)$ of $K$ is such that, for some $n$, 
the subspace $\cH^{K_n}$ is cyclic, then it is tame. 
\end{remark}

\begin{definition} {\rm Assume that the group $K$ 
is the union of an increasing 
sequence of subgroups $(K(n))_{n \in \N}$. We say that the subgroups 
$K(n)$ are {\it well-complemented} by 
the decreasing sequence $(K_n)_{n \in \N}$ of subgroups of $K$ if 
$K_n$ commutes with $K(n)$ for every~$n$ and $\bigcap_{n \in \N} K_n = \{\1\}$. 
For $k \in K$ and $n \in \N$, we then 
find an $m > n$ with $k \in K(m)$. Then 
$k K_m k^{-1} = K_m \subeq K_n,$ 
so that (S1-3) are satisfied and the groups $(K_n)_{n \in\N}$ define a group topology 
on $K$. }
\end{definition}

\begin{example} \mlabel{ex:3.4} 
(a) If $K = \oplus_{n = 1}^\infty F_n$ is a direct sum of subgroups $(F_n)_{n \in \N}$, 
then the subgroups $K(n) := F_1 \times \cdots \times F_n$ are well-complemented 
by the subgroups $K_n := \oplus_{m > n} F_m$. 

(b) If $K = \U(\infty,\K)_0 = \bigcup_{n = 1}^\infty \U(n,\K)_0$ is the canonical 
direct limit of the compact groups $\U(n,\K)_0$, then the subgroups 
$K(n) := \U(n,\K)_0$ are well-complemented by the subgroups 
\[ K_n := \{ g \in K \: (\forall j \leq n)\ ge_j= e_j\}.\] 
\end{example}

\subsection{Tame representations of $\U(\infty, \K)$} 
\mlabel{subsec:3.2}

Let $\cH$ be an infinite-dimensional separable Hilbert space 
over $\K \in \{\R,\C,\H\}$ and $(e_j)_{j \in \N}$ an orthonormal basis of~$\cH$. 
Accordingly, we 
obtain a natural dense embedding $\U(\infty,\K) \into \U_\infty(\cH)$, 
so that every continuous unitary representation of 
$\U_\infty(\cH)_0$ is uniquely determined by its restriction to the direct limit 
group $\U(\infty,\K)_0$. Olshanski's approach to the classification is based 
on an intrinsic characterization of those representations of the direct limit group 
$\U(\infty,\K)_0$ that extend to $\U_\infty(\cH)_0$. 
It turns out that these are precisely 
the tame representations (Theorem~\ref{thm:samerep}). 
This is complemented by the discrete decomposition 
and the classification of the irreducible ones 
(Theorem~\ref{thm:class}). 

In the following we write 
$K := \U_\infty(\cH)_0$ for the identity component of $\U_\infty(\cH)$ 
(which is connected for $\K = \C,\H$, but not for $\K = \R$),  
and $K(n) := \U(n,\K)_0 \cong \U(\cH(n))_0$ for $n \in \N$, 
where $\cH(n) = \Spann \{ e_1,\ldots, e_n\}$. 
For $n \in \N$, the stabilizer of 
$e_1, \ldots, e_n$ in $K$ is denoted $K_n$, and we likewise 
write $K(m)_n := K(m) \cap K_n$. 
We also write $K(\infty) := \U(\infty,\K)_0 \cong \indlim 
\U(n,\K)_0$ for the direct limit of the groups $\U(n,\K)_0$. 

We now turn to the classification of the continuous unitary representations 
of~$K$. We start with an application of 
Proposition~\ref{prop:1.6}. 

\begin{lemma} \mlabel{lem:p.1x} {\rm(Kirillov's Lemma)} 
Let $(\pi, \cH_\pi)$ be a continuous 
unitary representation of the Banach--Lie group $K = \U_\infty(\cH)_0$. 
If $\cH_\pi \not=\{0\}$, then there exists an $n \in \N$, such that the 
stabilizer $K_n$ of $e_1, \ldots, e_n$ has a non-zero fixed point. 
\end{lemma}

\begin{proof} We apply Proposition~\ref{prop:1.6}(b) 
with $G := K$, $G_n := K_n$ and $G(n) := K(n)$. 
Then 
$U_\eps := \{ g \in K \: \|g - \1\| < \eps\}$ 
provides the required basis of $\1$-neighborhoods in $G$ 
(Proposition~\ref{prop:1.3}(iii)). 
Condition (V1) follows from Proposition~\ref{prop:1.3}(iii), 
(V2)(1) is clear, 
and (V2)(2) follows from the fact that, 
for $1 < k_1 < \ldots < k_N$, 
elements $u_j \in K(k_j)_{k_{j-1}}$ act on pairwise orthogonal subspaces. 
\smartqed\qed\end{proof}

\begin{proposition} \mlabel{prop:3.8} 
Any continuous unitary representation $(\pi, \cH)$ 
of $K = \U_\infty(\cH)_0$ restricts to a tame 
representation of the subgroup $K(\infty) = \U(\infty, \K)_0$. 
\end{proposition}

\begin{proof} Let $\cH_0 \subeq \cH$ denote the maximal subspace on which the representation 
of $K(\infty)$ is tame, i.e., the  space of continuous vectors for the topology defined by 
the subgroups $K(\infty)_n$ (Remark~\ref{rem:2.3}). 
Lemma~\ref{lem:p.1x} implies that $\cH_0 \not=\{0\}$. 
If $\cH_0 \not=\cH$, then Lemma~\ref{lem:p.1x} implies the existence of non-zero continuous vectors 
in  $\cH_0^\bot$, which is a contradiction. 
\smartqed\qed\end{proof}

\begin{example} The preceding proposition does not extend to the 
non-connected group $\OO_\infty(\cH)$ which has $2$-connected components. 
The corresponding homomorphism 
\[ D \: \OO_\infty(\cH) \to \{ \pm 1 \} \] 
is non-trivial on all subgroups $\OO_\infty(\cH)_n$, $n \in \N$. 
\end{example}

\begin{lemma} \mlabel{lem:theta} Let $\cH$ be a $\K$-Hilbert space and 
$\cF \subeq \cH$ be a finite-dimensional subspace with 
$2 \dim \cF < \dim \cH$.\begin{footnote}{For $\K = \C,\H$, the 
condition $2\dim \cF \leq \dim \cH$ is sufficient.}   
\end{footnote}

We write $P_\cF \: \cH \to \cF$ for the orthogonal projection and
\[  C(\cF) := \{ A \in B(\cF) \: \|A\| \leq 1 \} \] 
for the semigroup of contractions on $\cF$. 
Then the map 
\[ \theta \: K = \U_\infty(\cH)_0 \to C(\cF), \quad 
\theta(g) = P_\cF g P_\cF^* \] 
is continuous, surjective and open. Its 
fibers are the double cosets of the pointwise stabilizer 
$K_\cF$ of~$\cF$. 

In particular, we obtain for $\cF = \Spann \{ e_1, \ldots, e_n\}$ a map 
\[ \theta \: K \to C(n,\K) := \{ X \in M(n,\K) \: \|X\| \leq 1\}, \quad 
\theta(k)_{ij} := \la k e_j, e_i \ra, \]  
which is continuous, surjective and open, and whose 
fibers are the double cosets $K_n k K_n$ for $k \in K$.
\end{lemma}

\begin{proof} (i) Surjectivity: For $C \in C(\cF)$, the operator 
\[ U_C := \pmat{ 
C &   \sqrt{\1- CC^*} \\ 
-\sqrt{\1 - C^*C} &  C^*} \in B(\cF \oplus \cF) \] 
is unitary. In view of $2\dim \cF \subeq \cH$, we have an isometric 
embedding $\cF \oplus \cF \into \cH$, and each unitary operator on 
$\cF \oplus \cF$ extends to $\cH$ by the identity on the orthogonal complement. 
To see that the resulting operator in contained in $K$, it remains to 
see that $\det U_C = 1$ if $\K = \R$. To verify this claim, we first observe 
that, for $U_1, U_2 \in \UU_n(\K)$, we have 
\[ U_{U_1 C U_2} 
= \pmat{U_1 & 0 \\ 0 & U_2^*} U_C\pmat{U_2 & 0 \\ 0 & U_1^*},\] 
which implies in particular that $\det U_C = \det U_{U_1 C U_2}$. We may 
therefore assume that $C$ is diagonal, and in this case the assertion 
follows from the trivial case where $\dim \cF = 1$.  
This implies that $\theta$ is surjective. 

(ii) $\theta$ separates the double cosets of $K_\cF$: 
We may w.l.o.g.\ assume that $e_1, \ldots, e_n$ span~$\cF$. 
First we observe that, for $m < 2n$, the subgroup 
$K_m$ acts transitively on spheres in 
$\cH(m)^\bot$.\begin{footnote}
{Our assumption implies that $\dim \cH \geq 2$. This claim 
follows from the case where $\cH = \K^2$. 
Using the diagonal inclusion $\U(1,\K)^2 \into \U(2,\K)$, it suffices to 
consider vectors with real entries, which reduces the problem 
to the transitivity of the action of $\SO(2,\R)$ on the unit circle. 
Since the trivial group $\SO(1,\R)$ does not act trivially on 
$\bS^0 = \{\pm 1\}$, it is here where we need that 
$2\dim \cF < \dim \cH$.}
\end{footnote}

Suppose that $\theta(k) = \theta(k')$, i.e., that the first 
$n$ components of the vectors $ke_j$ and $k' e_j$, $j =1,\ldots. n$, 
coincide. Let $P \: \cH \to \cF^\bot$ 
denote the orthogonal projection. 
Then $\|P k e_1\| = \|Pk' e_1\|$, so that the argument in the preceding 
paragraph shows that there exists a $k_1 \in K_n$ 
with $k_1 Pke_1 = Pk' e_1$. This implies that $k_1 k e_1 = k' e_1$. 
Replacing $k$ by $k_1 k$, we may now assume that $k e_1 = k' e_1$. 
Then $\|P ke_2\| = \|P k' e_2\|$ and the scalar products of 
$P k e_2$ and $Pk' e_2$  with $P k e_1$ coincide. 
We therefore find an element $k_2 \in K_n$ fixing $P k e_1$, hence also 
$ke_1$,  and satisfying $k_2 P ke_2 = P k' e_2$, i.e., 
$k_2 k e_2 = k' e_2$. Inductively, we thus obtain 
$k_1,\ldots, k_n \in K_n$ with 
$k_n \cdots k_1 k e_j = k' e_j$ for $j = 1,\ldots, n$, and this 
implies that $k' \in k_n \cdots k_1 k K_n \subeq K_n k K_n$. 

(iii) It is clear that $\theta$ is continuous. To see that it is open, 
let $O \subeq K$ be an open subset. Then $\theta(O) = \theta(K_n O K_n)$, so that we 
may w.l.o.g.\ assume that $O = K_n O K_n$. From (i) and (ii) it follows that every 
$K_n$-double coset intersects $K(2n+1)$, so that 
$\theta(O) = \theta(O \cap K(2n+1))$. Therefore it is enough to observe that the 
restriction of $\theta$ to $K(2n+1)$ is open, which follows from the 
compactness of $K(2n+1)$ and the fact that 
$\theta\res_{K(2n+1)} \: K(2n+1) \to C(n,\K)$ is a quotient map. 
\smartqed\qed\end{proof}

For a continuous unitary representation 
$(\pi, \cH_\pi)$ of $K$, let $V := \cH_\pi^{K_n}$ denote the subspace of $K_n$-fixed vectors, 
$P \: \cH_\pi \to V$ be the orthogonal projection and 
$\pi_V(g) := P^* \pi(g)P$. Then $\pi_V$ is a $B(V)$-valued continuous positive definite 
function and Lemma~\ref{lem:theta} implies that we obtain a well-defined continuous map 
\[ \rho \: C(n,\K) \to B(V), \quad 
\rho(\theta(k)) := \pi_V(k) \quad \mbox{ for } \quad k \in K.\] 
The operator adjoint $*$ turns $C(n,\K)$ into an involutive semigroup, and we 
obviously have $\pi_V(k)^* = \pi_V(k^*)$. 

Olshanski's proof of the following lemma is based on the fact 
that the projection of the invariant probability 
measure on $\bS^n$ to an axis for $n \to \infty$ to the 
Dirac measure in $0$. 

\begin{lemma} {\rm(\cite[Lemma~1.7]{Ol78})} \mlabel{lem:3.10} 
The map $\rho$ is a continuous $*$-rep\-re\-sen\-ta\-tion of the involutive 
semigroup $C(n,\K)$ by contractions satisfying $\rho(\1) = \1$. 
\end{lemma}

Using Zorn's Lemma, we conclude that $\pi$ is a direct sum of subrepresentations 
for which the subspace of $K_n$-fixed vectors is cyclic for some $n \in \N$. 
We may therefore 
assume that $V = (\cH_\pi)^{K_n}$ is cyclic in $\cH_\pi$. Then 
the representation $\pi$ is equivalent to the 
GNS-representation of $K$, defined by the positive definite function $\pi_V$ 
(Remark~\ref{rem:gns}). 
Since the subspace $V = (\cH_\pi)^{K_n}$ is obviously invariant under the commutant 
$\pi(K)'$ of $\pi(K)$, the cyclicity of $V$ implies that we have an injective map 
\[ \pi(K)' \to \rho(C(n,\K))' \subeq B(V)\] 
which actually is an isomorphism because $\pi_V(K)= \rho(C(n,\K))$ 
is a semigroup 
(Proposition~\ref{prop:b.1}).

Therefore the structure of $\pi$ is completely encoded in the representation $\rho$ 
of the semigroup $C(n,\K)$. We therefore have to understand the 
$*$-rep\-re\-sen\-tations $(\rho,V)$ of $C(n,\K)$ for which 
the $B(V)$-valued function $\rho\circ\theta \: K \to B(V)$ 
is positive definite. 

\begin{definition}
{\rm We call a $*$-representation $(\rho,V)$ of $C(n,\K)$ 
{\it $\theta$-positive} if the corresponding function 
$\rho \circ \theta \: K \to B(V)$ is positive definite. }
\end{definition}

If $\rho$ is $\theta$-positive, then we obtain a 
continuous unitary GNS-representation 
$(\pi_\rho, \cH_\rho)$ of $K$ containing a $K$-cyclic 
subspace $V$ such that the orthogonal projection 
$P \: \cH_\rho \to V$ satisfies 
$P\pi_\rho(g) P^* = \rho(\theta(g))$ for $g \in K$ (cf.\ Remark~\ref{rem:gns}). 
The following lemma shows that 
we can recover $V$ as the space of $K_n$-fixed vectors in $\cH_\rho$. 

\begin{lemma} \mlabel{lem:recover} $(\cH_\rho)^{K_n} = V$. 
\end{lemma}

\begin{proof} Since $\rho\circ \theta$ is $K_n$-biinvariant, 
the subspace $V$ consists of $K_n$-fixed vectors because 
$K$ acts in the corresponding subspace $\cH_{\rho \circ \theta} \subeq V^K$ 
by right translations (cf.\ Remark~\ref{rem:gns}). 
Let $W := (\cH_\rho)^{K_n}$ and $Q \: \cH_\rho \to W$ denote 
the corresponding orthogonal projection. 
Then $\nu(\theta(g)) := Q \pi_\rho(g) Q^*$ defines a contraction representation 
$(\nu,W)$ of $C(n,\K)$ (Lemma~\ref{lem:3.10}) and, for $s \in C(n,\K)$, we have
$\rho(s) = P \nu(s)P^*$. 

In $\cH$ the subspace $V$ is $K$-cyclic. Therefore the subspaces 
$Q \pi(g) V$, $g \in K$, span $W$. In view of 
$Q\pi(g)V = \nu(\theta(g)) V$, this means that 
$V \subeq W$ is cyclic for $C(n,\K)$. Now Remark~\ref{rem:b.6} implies that 
$V = W$. 
\smartqed\qed\end{proof}

We subsume the results of this subsection in the following proposition. 

\begin{proposition} 
Let $(\rho,V)$ be a continuous $\theta$-positive 
$*$-representation of $C(n,\K)$ by contractions and 
$\phi := \rho \circ \theta$. Then the corresponding GNS-representation 
$(\pi_\phi, \cH_\phi)$ of $K$ is continuous with cyclic subspace 
$V \cong (\cH_\phi)^{K_n}$ and $(\pi_\phi)_V = \phi = \rho \circ \theta$. 
This establishes a one-to-one correspondence of 
$\theta$-positive continuous $*$-representation of $C(n,\K)$ 
and continuous unitary representations $(\pi, \cH_\pi)$ of $K$ generated by 
the subspace $(\cH_\pi)^{K_n}$ of $K_n$-fixed vectors. This correspondence 
preserves direct sums of representations. 
\end{proposition}

\subsection{$\theta$-positive representations of $C(n,\K)$} 
\mlabel{subsec:3.3}

For $\K = \R,\H$, let $Z := ]0,1]\1 \subeq C(n,\K)$ be the central subsemigroup 
of real multiples of~$\1$. Then the continuous bounded characters
of $Z$ are of the form $\chi_s(r) := r^s$, $s \geq 0$. 
Any continuous $*$-representation $(\pi, V)$ of $Z$ by contractions 
determines a spectral measure $P$ on $\hat Z := \R_+$ satisfying 
$\pi(r\1) = \int_0^\infty r^s\, dP(s)$ 
(cf.\ \cite{BCR84}, \cite[VI.2]{Ne00}). 

For $\K = \C$, the subsemigroup $Z := \{ z \in \C^\times \1 \: |z| \leq 1\} 
\cong\,\, ]0,1] \times \T$ is also central in $C(n,\C)$. 
Its continuous bounded characters are of the form 
$\chi_{s,n}(re^{it}) := r^s e^{int}$, $s \geq 0, n \in \Z$. 
Accordingly, continuous $*$-representation of $Z$ by contractions 
correspond to spectral measures on $\hat Z := \R_+ \times \Z$. 

Let $(\rho, V)$ be a continuous (with respect to the weak operator topology on $B(V)$) 
$*$-representation of $C(n,\K)$ by contractions. Since the spectral projections 
for the restriction $\rho_Z := \rho\res_Z$ lie in the commutant of 
$\rho(C(n,\K))$, the representation $\rho$ is a 
direct sum of subrepresentations for which the support of the spectral measure of 
$\rho_Z$ is a compact subset of $\hat Z$. We call these representations 
{\it centrally bounded}. 
Then the operators $\rho(r\1)$ are invertible for $r > 0$, and this 
implies that 
\[ \hat\rho(r M) := \rho(r^{-1} \1)^{-1} \rho(M) \quad \mbox { for } \quad 
r > 1, M \in C(n,\K), \] 
yields a well-defined extension $\hat\rho$ of $\rho$ to a continuous 
$*$-representation of the multiplicative $*$-semigroup 
$(M(n,\K),*)$ on $V$. 
For a more detailed analysis of the decomposition theory, we may therefore restrict 
our attention to centrally bounded representations. Decomposing further as a 
direct sum of cyclic representations, it even suffices to consider separable 
centrally bounded representations. 

The following proposition contains the key new points compared with 
Olshanski's approach in \cite{Ol78}. Note that it is very close to 
the type of reasoning used in \cite{JN13} for the classification 
of the bounded unitary representations of $\SU_2(\cA)$. 

\begin{proposition} \mlabel{prop:new} For every centrally bounded 
contraction representation $(\rho,V)$ of $C(n,\K)$, the following 
assertions hold: 
\begin{description}
\item[\rm(i)] The restriction of $\hat\rho$ to $\GL(n,\K)$ is a 
norm-continuous representation whose differential 
$\dd\hat\rho \: \gl(n,\K) \to B(V)$ is a representation of the Lie algebra 
$\gl(n,\K)$ by bounded operators. 
\item[\rm(ii)] If $\rho$ is $\theta$-positive, then 
$\hat\rho$ is real analytic on $M(n,\K)$ 
and extends to a holomorphic semigroup representation of the complexification 
\[  M(n,\K)_\C \cong 
\begin{cases}
M(n,\C)  & \text{ for } \K = \R\\ 
M(n,\C) \oplus M(n,\C) \cong 
B(\C^n) \oplus B(\oline{\C^n})&  \text{ for } \K = \C \\ 
M(n,M(2,\C)) \cong M(2n,\C)& \text{ for } \K = \H.  \\
\end{cases} \]
\end{description}
\end{proposition}

\begin{proof} (i) We have already seen that $\hat\rho \: M(n,\K) \to B(V)$ 
is locally bounded and continuous. Hence it restricts to a locally bounded 
continuous representation of the involutive Lie group 
$(\GL(n,\K),*)$. 
Integrating this representation to the convolution algebra $C^\infty_c(\GL(n,\K))$, 
we see that 
the subspace $V^\infty$ of smooth vectors is dense. 
For the corresponding derived representation 
\[ \dd\hat\rho \:  \gl(n,\K) \to \End(V^\infty) \] 
our construction immediately implies that the operator 
$\dd\hat\rho(\1)$ is bounded and 
$\dd\hat\rho(X) \geq 0$ for $X = X^*\geq 0$ because we started 
with a contraction representation of $C(n,\K)$. 
From $X \leq \|X\| \1$ we also derive 
$\dd\hat\rho(X) \leq \|X\| \dd\hat\rho(\1)$, so that 
$\dd\rho(X)$ is bounded. 
As $\fu(n,\K) \subeq \fz(\gl(n,\K)) + [\Herm(n,\K), \Herm(n,\K)]$, we conclude that 
$\dd\hat\rho$ is a $*$-representation by bounded operators on the 
Hilbert space~$V$. 

For $X \in \gl(n,\K)$, we then have the relation 
\[ \hat\rho(\exp X) = e^{\dd\hat\rho(X)} \quad \mbox{ for } \quad X \in \gl(n,\K).\] 
This implies that $\hat\rho \: \GL(n,\K) \to \GL(V)$ is norm-continuous. 

(ii) Now we assume that $\phi := \rho \circ \theta \: K \to B(V)$ 
is positive definite. Since $\theta(\1) = \1$, there exists an open 
$\1$-neighborhood $U \subeq K$ with $\theta(U) \subeq \GL(n,\K)$. 
For $k \in U$ we then have 
$\phi(k) = \hat\rho(\theta(k))$, and since the representation 
$\hat\rho$ of $\GL(n,\K)$ is norm continuous, hence analytic, 
$\phi$ is analytic on $U$. 
Now Theorem~\ref{thm:locana} implies that $\phi$ is analytic. 

Let $\Omega := \{ C \in C(n,\K) \: \|C\|  < \1\}$ denote the interior 
of $C(n,\K)$. On this domain we have an analytic cross section of $\theta$, 
given by 
\[ \sigma(C) := \pmat{ C &   \sqrt{\1- CC^*} & 0 \\ 
-\sqrt{\1 - C^*C} &  C^* & 0 \\ 
0 & 0 & \1}. \] 
Now $\phi(\sigma(C)) = \rho(\theta(\sigma(C))) = \rho(C)$ for 
$C \in \Omega$ implies that $\rho\res_{\Omega}$ is analytic. 
From $\hat\rho(r C) = \hat\rho(r) \rho(C)$ for 
$r > 0$ it now follows that $\hat\rho \: M(n,\K) \to B(V)$ 
is analytic. 

It remains to show that $\hat\rho$ extends to a holomorphic 
map on $M(n,\K)_\C$. First, the analyticity of $\hat\rho$ implies 
for some $\eps > 0$ the existence of a holomorphic map 
$F$ on $B_\eps := \{ C \in M(n,\K)_\C \: \|C\| < \eps\}$ with 
$F(C) = \hat\rho(C)$ for $C \in M(n,\K) \cap B_\eps$. 
This map also satisfies 
$F(r C) = \hat\rho(r\1) F(C)$ for $r < 1$, which implies that 
$F$ extends to a holomorphic map on $r^{-1} B_\eps = B_{r^{-1}\eps}$ for every 
$r > 0$. This leads to the existence of a holomorphic extension 
of $\hat\rho$ to all of $M(n,\K)_\C$. 
That this extension also is multiplicative follows immediately 
by analytic continuation. 
\smartqed\qed\end{proof}

\begin{theorem} \mlabel{thm:class-contrac} 
{\rm(Classification of irreducible $\theta$-positive representations)} 
Put $\cF := \K^n$. Then all  
irreducible continuous $\theta$-positive representations 
of $C(\cF) \cong C(n,\K)$ are of the form 
\[\begin{cases}
\bS_\lambda(\cF_\C) \subeq (\cF_\C)^{\otimes N}, 
& \text{ for } \K = \R,  \\ 
\bS_\lambda(\cF) \otimes 
\bS_\mu(\oline\cF) \subeq 
\cF^{\otimes N} \otimes \oline\cF^{\otimes M}, 
  & \text{ for } \K = \C,  \\
\bS_\lambda(\cF^\C) \subeq (\cF^\C)^{\otimes N}, 
  & \text { for } \K = \H,  \\
\end{cases} \]
where $\lambda \in \Part(N,n), \mu \in \Part(M,n)$.
\end{theorem}

\begin{proof} Let $(\rho,V)$ be an irreducible 
$\theta$-positive representation of $C(n,\K)$. 
Then $\rho(Z) \subeq \C \1$ by Schur's Lemma, 
so that $\rho$ is in particular centrally bounded 
and extends to a holomorphic representation 
$\hat\rho \: M(n,\K)_\C \to B(V)$ (Proposition~\ref{prop:new}). 

For $\K = \R,\H$, the center of $M(n,\K)_\C$ is $\C \1$. 
Since the only holomorphic multiplicative maps 
$\C \to\C$ are of the form $z \mapsto z^N$ for some $N \in \N_0$, 
it follows that $\hat\rho(z\1) = z^N \1$ for $z \in \C$. 
We conclude that the holomorphic map $\hat\rho$ is homogeneous of 
degree $N$. Hence there exists a linear map 
\[ \tilde\rho \: S^N(M(n,\K)_\C) \to B(V) 
\quad \mbox{ with } \quad 
\tilde\rho(A^{\otimes N}) = \hat\rho(A), \quad A \in M(n,\K)_\C.\] 
The multiplicativity of $\hat\rho$ now implies 
that $\tilde\rho$ is multiplicative, hence a 
representation of the finite-dimensional algebra 
$S^N(M(n,\K)_\C)$. 

For $\K = \R$, we have $M(n,\R)_\C = M(n,\C)$, and
\[ S^N(M(n,\C)) =  (M(n,\C)^{\otimes N})^{S_N} 
\cong M(nN,\C)^{S_N} \cong B((\C^n)^{\otimes N})^{S_N}.\] 
We conclude that $S^N(M(n,\C))$ 
is the commutant of $S_N$ in $M(nN,\C)$, and by 
Schur--Weyl theory, this algebra can be identified with the 
image of the group algebra $\C[\GL(n,\C)]$ in $B((\C^n)^{\otimes N})$. 
Therefore its irreducible representations are parametrized by the set 
$\Part(N,n)$ of partitions of $N$ into at most $n$ summands. 
This completes the proof for $\K =\R$. 
For $\K = \H$, we have the same picture because $M(n,\H)_\C \cong M(2n,\C)$. 

For $\K = \C$, $Z(M(n,\C)_\C) \cong \C^2$, and the inclusion 
of $Z(M(n,\C)) = \C \1$ has the form $z \mapsto (z,\oline z)$. 
Hence there exist $N,M \in \N_0$ with $\rho_Z(z\1) = z^N \oline z^M \1$. 
Therefore the restriction of $\hat\rho$ to the first factor 
is homogeneous of degree $N$ and the restriction to the second 
factor of degree~$M$. This leads to a representation 
of the algebra 
\[ S^{N,M}(M(n,\C)) := S^N(M(n,\C)) \otimes S^M(M(n,\C)),\] 
so that the same arguments as in the real case apply. 
\smartqed\qed\end{proof}

Now that we know all irreducible $\theta$-positive representations, 
we ask for the corresponding decomposition theory. 

\begin{theorem} \mlabel{thm:deco} 
Every continuous $\theta$-positive $*$-representation of $C(n,\K)$ by contractions 
is a direct sum of irreducible ones, and these are 
finite-dimensional.  
\end{theorem}

\begin{proof} We have already seen that $\rho$ decomposes into a direct 
sum of centrally bounded representations. We may therefore 
assume that $\rho$ is centrally bounded, so that 
$\rho$ extends to a holomorphic representation $\hat\rho$ 
of $M(n,\K)_\C$ (Proposition~\ref{prop:new}). 
In the $C^*$-algebra $\cA := M(n,\K)_\C$, every holomorphic function 
is uniquely determined by its restriction 
to the unitary group $\U(\cA)$, which is a totally real submanifold. 
Therefore a closed subspace $W \subeq V$ is invariant 
under $\rho(C(n,\K))$ if and only if it is invariant 
under $\hat\rho(\U(\cA))$. 
Since the group $\U(\cA)$ is compact, the assertion now 
follows from the classical fact that unitary representations 
of compact groups are direct sums of irreducible~ones. 
\smartqed\qed\end{proof}

\subsection{The Classification Theorem} 
\mlabel{subsec:3.4}

We are now ready to prove the 
Kirillov--Olshanski Theorem (\cite{Ki73}, \cite{Ol78}). 

\begin{theorem} \mlabel{thm:class} 
{\rm(Classification of the representations of $\U_\infty(\cH)_0$)} 
Let $\cH$ be an infinite-dimensional separable $\K$-Hilbert space.  \\
{\rm(a)} The irreducible continuous unitary representations of 
$\U_\infty(\cH)_0$ are 
\[ \begin{cases}
\bS_\lambda(\cH_\C) \subeq (\cH_\C)^{\otimes N}, 
 & \text{ for } \K = \R,  \\ 
\bS_\lambda(\cH) \otimes 
\bS_\mu(\oline\cH) \subeq 
\cH^{\otimes N} \otimes \oline\cH^{\otimes M}, 
& \text{ for } \K = \C,  \\
\bS_\lambda(\cH^\C) \subeq (\cH^\C)^{\otimes N}, 
& \text { for } \K = \H,  \\
\end{cases} \]
where $\lambda \in \Part(N), \mu \in \Part(M)$. \\
{\rm(b)} Every continuous unitary representation of $\U_\infty(\cH)_0$ 
is a direct sum of irreducible ones. \\
{\rm(c)} Every continuous unitary representation of $\U_\infty(\cH)_0$ 
extends uniquely to a continuous 
unitary representations of the full unitary group $\U(\cH)_s$, 
endowed with the strong operator topology. 
\end{theorem}

\begin{proof} (a) In Theorem~\ref{thm:class-contrac} we have 
classified the irreducible $\theta$-positive 
representations of $C(n,\K)$. 
The corresponding representations $(\pi,\cH_\pi)$ of $K = 
\U_\infty(\cH)_0$ can now be determined 
rather easily. Since the passage from $\rho$ to $\pi$ preserves direct sums, 
we consider the representations $\rho_{N}$ of $C(n,\K)$ 
on $\cF_\C^{\otimes N}$ for $\K =\R$, on 
$(\cF^\C)^{\otimes N}$ for $\K = \H$, and the representation $\rho_{N,M}$ on 
$\cF^{\otimes N} \otimes \oline\cF^{\otimes M}$ for $\K = \C$. 

We likewise have unitary representations 
$\pi_{N}$ of $K$ on 
$(\cH_\C)^{\otimes N}$ for $\K = \R$, on $(\cH^\C)^{\otimes N}$ for $\K = \H$, and 
a representation 
$\pi_{N,M}$ on $\cH^{\otimes N} \otimes \oline{\cH}^{\otimes M}$ for $\K = \C$. 
These are bounded continuous representations of $K$. 

For $K = \R, \H$, the space of $K_n$-fixed vectors in $(\cH_\C)^{\otimes N}$ obviously contains 
$(\cF_\C)^{\otimes N}$ and by considering the action of the subgroup of 
diagonal matrices, we obtain 
the equality $(\cF_\C)^{\otimes N} = ((\cH_\C)^{\otimes N})^{K_n}$. 
Therefore the representation $\pi_{N}$ corresponds to the representation 
$\rho_{N}$ of $C(n,\K)$. 
A similar argument shows that, for $\K = \C$, the $K$-representation 
$\pi_{N,M}$ corresponds to $\rho_{N,M}$. 

Since the representations $\rho_N$ and $\rho_{N,M}$ decompose into finitely many 
irreducible pieces, the representations $\pi_N$ and $\pi_{N,M}$ decompose in 
precisely the same way. For $\K =\R,\H$, we thus obtain the Schur modules 
$\bS_\lambda(\cH_\C)$ and 
$\bS_\lambda(\cH^\C)$ with 
$\lambda \in \Part(N)$, respectively. 
For $\K= \C$, we obtain the tensor products 
$\bS_\lambda(\cH) \otimes \bS_\mu(\oline\cH)$ with 
$\lambda \in \Part(N)$ and $\mu \in \Part(M)$. 

Here the restriction to partitions consisting of at most $n$ summands 
corresponds to the $K$-invariant 
subspace generated by the $K_n$-fixed vectors. This subspace is 
proper if $n$ is small. 

(b) From Theorem~\ref{thm:deco} we know 
that $\theta$-positive contraction representations of $C(n,\K)$ 
are direct sums of irreducible ones. 
This implies that all continuous unitary representations of $K$ 
are direct sums of irreducible ones. Since the correspondence between 
$\pi$ and $\rho$ leads to isomorphic commutants, the 
irreducible subrepresentations of $\pi$ and the corresponding subrepresentations 
of $\rho$ have the same multiplicities. 

(c) The assertion is trivial for the irreducible representations of $K$ 
described under (a). 
Since $\U_\infty(\cH)_0$ is dense in 
$\U(\cH)$ with respect to the strong operator topology\begin{footnote}{This follows from the fact that $\U_\infty(\cH)_0$ 
acts transitively on the finite orthonormal systems in $\cH$.}\end{footnote}, 
this extension is unique and generates the same von Neumann algebra. 
\smartqed\qed\end{proof}

\begin{remark} (Representations of $\OO_\infty(\cH)$) 
The above classification can easily be extended to the 
non-connected group $\OO_\infty(\cH)$ (for $\K = \R$). 
Here the existence of a canonical extension $\oline\pi_\lambda$ of 
every irreducible representations $\pi_\lambda$ of 
$\SO_\infty(\cH) := \OO_\infty(\cH)_0$ to $\OO(\cH)$ implies that there exist precisely two extensions 
that differ by a twist with the canonical 
character $D \: \OO_\infty(\cH) \to \{\pm 1\}$ corresponding to the 
determinant. 

For a general continuous unitary representations of 
$\OO_\infty(\cH)$, it follows that all $\SO_\infty(\cH)$-isotypic 
subspaces are invariant under $\OO_\infty(\cH)$, hence of the form 
$\cM_\lambda \otimes \cH_\lambda$, where $\OO_\infty(\cH)$ acts 
by $\eps \otimes \oline\pi_\lambda$ and 
$\eps$ is a unitary representation of the 
$2$-element group $\pi_0(\OO_\infty(\cH))$, i.e., defined by a 
unitary involution. 

In particular, all continuous unitary representations 
of $\OO_\infty(\cH)$ are direct sums of irreducible ones, which are 
of the form $\oline\pi_\lambda$ and $D \otimes \oline\pi_\lambda$. 
Here the first type extends to the full orthogonal group $\OO(\cH)$, 
whereas the second type does not. 
\end{remark}

\begin{remark} (Extension to overgroups) 
\mlabel{rem:3.11} (cf.\ \cite[\S 1.11]{Ol84})
Let $\cH$ be an infinite-dimensional separable $\K$-Hilbert space 
and $K := \U_\K(\cH)$. We put 
\[  \cH^\sharp := 
\begin{cases} 
\cH_\C & \text{ for } \K = \R \\ 
\cH \oplus \oline\cH & \text{ for } \K = \C \\ 
\cH^\C & \text{ for } \K = \H. 
\end{cases}\]
For each $N \in \N$ we obtain a norm 
continuous representation 
\[ \pi_N \:  (B(\cH^\sharp), \cdot) \to B((\cH^\sharp)^{\otimes N}), \quad 
\pi_N(A) := A^{\otimes N}\] 
of the multiplicative semigroup $(B(\cH^\sharp),\cdot)$ whose restriction 
to $\U(\cH^\sharp)$ is unitary. 

We collect some properties of this representation: 

(a) Let $C(\cH^\sharp) = \{ S \in B(\cH^\sharp) \: \|S\| \leq 1\}$ 
denote the closed subsemigroup of contractions. Then $C(\cH^\sharp)$ 
is a $*$-subsemigroup 
of $B(\cH^\sharp)$ and $\pi_N\res_{C(\cH^\sharp)}$ is continuous with respect to the weak operator 
topology. In fact, $\pi_N(C(\cH^\sharp))$ consists of contractions, and for the 
total subset of vectors of the form 
$v := v_1 \otimes \cdots \otimes v_N$, $w := w_1 \otimes \cdots \otimes w_N$, the 
matrix coefficient $S \mapsto \la \pi_N(S)v,w \ra = \prod_{j = 1}^N \la Sv_j, w_j \ra$ 
is continuous. 

(b) $K := \U_\K(\cH)$ 
is dense in $C_\K(\cH)$ with respect to the weak operator topology. 
It suffices to see that, for every contraction $C$ on a finite-dimensional subspace 
$\cF \subeq \cH$,  there exists a unitary operator $U \in \U_\K(\cH)$ with 
$P_\cF U P_\cF^* = C$, where $P_\cF \: \cH \to \cF$ is the orthogonal projection. 
Since $\cF \oplus \cF$ embeds isometrically into $\cH$, this follows from the fact that 
the matrix 
\[ U := \pmat{ 
C &   \sqrt{\1- CC^*} \\ 
-\sqrt{\1 - C^*C} &  C^*} \in M_2(B_\K(\cF)) = B_\K(\cF \oplus \cF) \] 
is unitary and satisfies $P_\cF U P_\cF^* = C$ (Lemma~\ref{lem:theta}). 

(c) Combining (a) and (b) implies that $\pi_N(C_\K(\cH)) \subeq \pi_N(K)''$, and hence that 
$\pi_N(B_\K(\cH)) = \bigcup_{\lambda > 0} \lambda^N \pi_N(C_\K(\cH)) \subeq \pi_N(K)''$. 
For the corresponding Lie algebra representation 
\[ \dd\pi_N \: \gl_\K(\cH) \to B(\cH^{\otimes N}), \quad 
\dd\pi_N(X) := \sum_{j = 1}^N \1^{\otimes (j-1)} \otimes X \otimes 
\1^{\otimes (N-j)},\] 
this implies that $\dd\pi_N(\gl_\K(\cH)) \subeq \pi_N(K)''$, and hence also that 
$\dd\pi_N(\gl_\K(\cH))_\C \subeq \pi_N(K)''$. The connectedness of the group 
$K^\sharp$ (Lemma~\ref{lem:k-grp}) 
now implies that $\pi_N(K^\sharp) \subeq \pi_N(K)''$. 
Since the subgroup $K^\sharp_\infty$, consisting of those elements $g$ for which 
$g - \1$ is compact, is strongly dense in $K^\sharp$ and the 
representation of $K^\sharp$ is continuous with respect to the strong 
operator topology, the representations $\pi_N$ of $K^\sharp$ thus decomposes 
into Schur modules, as described in Theorem~\ref{thm:class}(a). 

(d) The preceding discussion shows in particular that the 
representation $\pi_N$ of $K$ extends to the overgroup $K^\sharp$ without 
enlarging the corresponding von Neumann algebra. 
If $\rho$ is the corresponding representation of 
$C(n,\K) = C(\cF)$ on $V := (\cF^\sharp)^{\otimes N}$, where 
$\cF = \K^n$, then $\rho$ extends to a holomorphic 
representation $\hat\rho$ of $M(n,\K)_\C$ and 
the map $\theta \: K \to C(n,\K)$ likewise extends to a holomorphic map 
$\hat\theta \: B(\cH)_\C \to M(n,\K)_\C$. Now 
$\hat\rho \circ \hat\theta \: B(\cH)_\C \to B(V)$ is a holomorphic 
positive definite function corresponding to the representation 
of $(B(\cH)_\C, \cdot)$ on $(\cH^\sharp)^{\otimes N}$ whose restriction 
yields a unitary representation of the unitary group 
$\U(\cH)^\sharp$ of $B(\cH)_\C$. 
\end{remark}

We conclude this subsection with the following converse to 
Proposition~\ref{prop:3.8}. 
\begin{theorem} \mlabel{thm:samerep}
A unitary representation of $\U(\infty,\K)_0$ is tame if and only 
if it extends to a continuous unitary representation of $\U_\infty(\cH)_0$ for 
$\cH = \ell^2(\N,\K)$. 
\end{theorem}

\begin{proof} We have already seen in Proposition~\ref{prop:3.8} 
that every continuous unitary representation of $K= \U_\infty(\cH)_0$ restricts 
to  a tame representation of $K(\infty) = \U(\infty,\K)_0$. 

Suppose, conversely, that $(\pi, \cH_\pi)$ is a tame unitary representation 
of $K(\infty)$. Then the same arguments as for $K$ 
imply that it is a direct sum of representations generated by the subspace 
$V = (\cH_\pi)^{K(\infty)_n}$ and we obtain a representation 
$(\rho,V)$ of $C(n,\K)$ for which 
$\rho \circ \theta$ is positive definite on $K(\infty)$. Since 
it is continuous and $K(\infty)$ is dense in $K$, it is also 
positive definite on $K$. Now the GNS construction, applied to 
$\rho \circ \theta$, yields the continuous extension of $\pi$ to $K$.
\smartqed\qed\end{proof}

\subsection{The inseparable case} 

In this subsection we show that Theorem~\ref{thm:class} 
extends to the case where $\cH$ is not separable.

\begin{theorem} \mlabel{thm:insep} Theorem~\ref{thm:class} also holds if $\cH$ is inseparable.
\end{theorem}

\begin{proof} (a) First we note that the Schur--Weyl decomposition 
\[ \cH^{\otimes N} \cong \bigoplus_{\lambda \in \Part(N)} \bS_\lambda(\cH) \otimes \cM_\lambda\]  
holds for any infinite-dimensional complex Hilbert space (\cite{BN12}) and that 
the spaces $\bS_\lambda(\cH)$ carry irreducible representations of $\U(\cH)$ 
which are continuous with respect to the norm topology and the 
strong operator topology on $\U(\cH)$. 

(b) To obtain the irreducible representations of 
$K := \U_\infty(\cH)_0$, we choose an orthonormal basis $(e_j)_{j \in J}$ of $\cH$ 
and assume that $\N = \{ 1,2,\ldots \}$ is a subset of $J$. Accordingly, we 
obtain an embedding $K(\infty) := \U(\infty,\K) \into \U_\infty(\cH)$ 
and define $K_n := \{ k \in K \: ke_j =e_j, j =1,\ldots,n\}$.
For a subset $M \subeq J$, we put 
$K(M) := \U_\infty(\cH_M)_0$, where $\cH_M \subeq \cH$ is the closed subspace 
generated by $(e_j)_{j\in M}$. 

(c) Kirillov's Lemma~\ref{lem:p.1x} is still valid in the inseparable case and 
Lemma~\ref{lem:3.10} follows from the separable case because 
$\theta(K) = \theta(K(\infty))$. 

(d) With the same argument as in Subsection~\ref{subsec:3.2}  it follows 
that $\pi$ is a direct sum of subrepresentations 
for which $(\cH_\pi)^{K_n}$ is cyclic. These in turn correspond to 
$\theta$-positive representations $(\rho,V)$ of $C(n,\K)$. 
We claim that $\rho \circ \theta$ is positive definite on $K(\N)$ 
if and only if it is positive definite on $K(M)$ for any countable 
subset with $\N \subeq M \subeq J$. In fact, there exists a unitary 
isomorphism $U_M \: \cH(\N) \to \cH(M)$ fixing $e_1,\ldots, e_n$. For 
$k \in K(M)$ we then have $\rho(\theta(k))=\rho(\theta(U_M^* k U_M))$, 
so that $\rho \circ \theta$ is positive definite on $K(M)$ if it is on~$K(\N)$. 

For every finite subset $F \subeq K$, there exists a 
countable subset $J_c \subeq J$ containing $\N$ such that 
$F$ fixes all basis elements $e_j$, $j \not\in J_c$. Therefore 
$\rho \circ \theta$ is positive definite on $K$ if and only if 
this is the case on $K(M)$ for every countable subset and this in turn 
follows from the positive definiteness on the subgroup $K(\N)$. 
We conclude that the classification of the unitary representations 
of $K$ is the same as for $K(\N)$. 
\smartqed\qed\end{proof}

\begin{remark} \mlabel{rem:insep} If $\cH$ is inseparable, then the classification implies that 
all irreducible unitary representations of 
$\U_\infty(\cH)_0$ are inseparable. In particular, all separable 
unitary representations of $\U_\infty(\cH)_0$ are trivial because they 
are direct sums of irreducible ones. 
\end{remark}

\begin{problem} It seems that the classification problem we dealt with 
in this section can be formulated in a more general context as follows. 
Let $\cA$ be a real involutive Banach algebra and 
$P \in \cA$ be a hermitian projection, so that we obtain a closed 
subalgebra $\cA_P := P \cA P$. On the unitary 
group $\U(\cA) := \{ A \in \cA \: A^* A = AA^* = \1\}$ we consider the map 
\[ \theta \: \U(\cA) \to C(\cA_P) := \{ A \in \cA_P \: \|A\| \leq 1\}, \quad 
\theta(g) := P g P.\] 
For which $*$-representations $(\rho,V)$ of the semigroup $C(\cA_P)$ 
is the function \break $\rho \circ \theta \: \U(\cA)_0 \to B(V)$ positive definite?  

For $\cA = B_\infty(\cH)$, the compact operators on the $\K$-Hilbert space 
$\cH$ and a finite rank projection $P$, this problem specializes to the determination 
of the $\theta$-positive representations of $C(n,\K)$. 

If $P$ is central, then $\theta$ is a $*$-homomorphism, so that 
$\rho \circ \theta$ is positive definite for any representation~$\rho$. 
\end{problem}

\section{Separable representations of $\U(\cH)$} \mlabel{sec:4}

In this section we show that, for the unitary group $\U(\cH)$ of a separable 
Hilbert space $\cH$, endowed with the norm topology, all separable representations are 
uniquely determined by their restrictions to the normal subgroup $\U_\infty(\cH)_0$. 
This result of Pickrell \cite{Pi88} 
extends the Kirillov--Olshanski classification to separable representations of~$\U(\cH)$. 

\subsection{Triviality of separable representations modulo compacts} 

Before we turn to the proof of Theorem~\ref{thm:p.1b}, we need a few preparatory 
lemmas. 

\begin{lemma} \mlabel{lem:p.1} Let $\cH$ be an infinite-dimensional Hilbert space, 
$(\pi, \cH_\pi)$ be a continuous unitary representation of $\U(\cH)$,  and 
$\cH = V \oplus V^\bot$ with $V \cong V^\bot$. 
Then $\cH_\pi^{\U(V)} \not=\{0\}$. 
\end{lemma}

\begin{proof} Put $\cH_1 := V$ and write $V^\bot$ as a Hilbert space direct 
sum $\hat\oplus_{j = 2}^\infty \cH_j$, where each $\cH_j$ is is isomorphic to $V$ or 
$\cH$. 
This is possible because  $|J| = |\N \times J|$ for every infinite set $J$. 
We claim that some $\U(\cH_j)$ has non-zero 
fixed points in $\cH_\pi$. Once this claim is proved, we choose 
$g \in \U(\cH)$ with $gV = \cH_j$. 
Then $\pi(g) \cH_\pi^{\U(V)} = \cH_\pi^{\U(\cH_j)} \not=\{0\}$ 
implies the assertion. 

For the proof we want to use Proposition~\ref{prop:1.6}. 
In $G := \U(\cH)$ we consider the basis of $\1$-neighborhoods given by 
$U_\eps := \{ g \in \U(\cH) \: \|g - \1\| < \eps\}$ and the subgroups 
$G_j := \U(\cH_j)$. 
Then the proof of Proposition~\ref{prop:1.3}(i) shows that 
there exists an $m \in \N$ with 
$G_j \subeq (U_\eps \cap G_j)^m$ for every $j$, which is (U1). 
It is also clear that (U2) is satisfied. Therefore the assertion follows 
from Proposition~\ref{prop:1.6}. 
\smartqed\qed\end{proof}

\begin{lemma} \mlabel{lem:p.2a} 
Let  $\cF \subeq \cH$ be a closed subspace of finite codimension. Then 
the natural morphism $\U(\cF) \to \U(\cH)/\U_\infty(\cH)$ is surjective, 
i.e., $\U(\cH) = \U_\infty(\cH) \U(\cF)$. 
\end{lemma}

\begin{proof} Since the groups $\U(\cH)$ and $\U(\cF)$ are connected 
(Proposition~\ref{prop:1.3}(i)), 
it suffices to show that their Lie algebras satisfy 
\[ \fu(\cH) = \fu(\cF) + \fu_\infty(\cH).\] 
Let $P \: \cH \to \cH$ be the orthogonal projection onto $\cF$. 
Then every $X \in \fu(\cH)$ can be written as 
\[ X = PXP + (\1-P)XP + X(\1-P),\] 
where $PXP \in \fu(\cF)$ and the other two summands are compact 
because $\1 - P$ has finite range. 
\smartqed\qed\end{proof}

\begin{lemma} \mlabel{lem:p.2} Let $(\pi, \cH_\pi)$ be a continuous 
unitary representation of $\U(\cH)$ with $\U_\infty(\cH)\subeq \ker \pi$ and 
$\cH = \hat\oplus_{j \in J} \cH_j$ with $\cH_j$ infinite-dimensional separable and 
$J$ infinite. Then $\bigcap_{j\in J} \cH_\pi^{\U(\cH_j)} \not=\{0\}$. 
\end{lemma}

\begin{proof} Let $V \subeq \cH$ be a closed subspace of the form 
$V = \oline{\sum_j V_j}$, where each $V_j \subeq \cH_j$ is a closed subspace 
of codimension $1$. 
Then $V^\bot \cong \ell^2(J,\C) \cong \cH \cong V$ because
$|J \times \N| = |J|$. According to Lemma~\ref{lem:p.2a}, we then have 
\[ \U(\cH_j) 
\subeq \U(V_j)\U_\infty(\cH_j)
\subeq \U(V)\U_\infty(\cH).\] 
In view of Lemma~\ref{lem:p.1}, $\U(V)$ has non-zero fixed points in 
$\cH_\pi$, 
and since $\U_\infty(\cH) \subeq \ker \pi$, any such fixed point is fixed 
by all the subgroups $\U(\cH_j)$. 
\smartqed\qed\end{proof}

From now on we assume that $\cH$ is separable. 

\begin{lemma} \mlabel{lem:p.3} Let $(\pi, \cH_\pi)$ be a continuous 
unitary representation of $\U(\cH)$ with $\U_\infty(\cH)\subeq \ker \pi$ 
and $g \in \U(\cH)$. 
If $1$ is contained in the essential spectrum of~$g$, i.e., 
the image of $g - \1$ in the Calkin algebra $B(\cH)/K(\cH)$ is not invertible, then 
$1$ is an eigenvalue of $\pi(g)$. 
\end{lemma}

\begin{proof} We choose an orthogonal decomposition $\cH = \hat\oplus_{n = 1}^\infty \cH_n$ 
into infinite-di\-men\-sio\-nal $g$-invariant subspaces of $\cH$ as follows. 

{\bf Case 1:} 
If $1$ is an eigenvalue of $g$ of infinite multiplicity, 
then we put $\cH_0 := \ker(\1 - g)^\bot$. If this space is infinite-dimensional, 
then we put $\cH_1 := \cH_0$, and if this is not the case, then we 
pick a subspace $\cH_0' \subeq \cH_0^\bot$ of 
infinite dimension and codimension and put $\cH_1 := \cH_0 \oplus \cH_0'$. 
We choose all other $\cH_n$, $n > 1$, such that 
$\cH_1^\bot = \hat\oplus_{n = 2}^\infty \cH_n$ and note that $\cH_1^\bot \subeq \ker(\1 -g)$. 

{\bf Case 2:} If $\ker(\1 - g)$ is finite-dimensional, then 
let $P_\eps \in B(\cH)$ 
be the spectral projection for $g$ corresponding 
to the closed disc of radius $\eps > 0$ about $1$. 
Then 
\[ g_\eps := P_\eps \oplus (\1-P_\eps)g\]  
satisfies $\|g_\eps - g\| \leq \eps$. 
The non-compactness of $g-\1$ implies that, if $\eps$ is small enough, 
then \[ g_\eps - \1 = 0 \oplus (\1 - P_\eps)(g - \1) \]  is non-compact, and hence that 
$P_\eps\cH$ has infinite codimension. 
Further $P_\eps\cH$ is infinite-dimensional because $1$ is an essential 
spectral value of $g$. 
Hence there exists a sequence 
$\eps_1 > \eps_2 > \ldots$ converging to $0$,  
for which the $g$-invariant subspaces 
\[ \cH_1 := (P_{\eps_1}\cH)^\bot 
\quad \mbox{ and } \quad 
\cH_j := P_{\eps_{j-1}}\cH \cap (P_{\eps_{j}}\cH)^\bot\] 
are infinite-dimensional.

In both cases, we consider $g_\eps$ as an element $(g_{\eps,n})$ of 
the product group \break $\prod_{n=1}^\infty \U(\cH_n) \subeq \U(\cH)$ satisfying 
$g_{\eps,n} = \1$ for $n$ sufficiently large. 
If $v \in \cH_\pi$ is a non-zero simultaneous fixed vector for 
the subgroups $\U(\cH_n)$ (Lemma~\ref{lem:p.2}), 
we obtain $\pi(g_\eps)v= v$ for every $\eps > 0$, and now 
$v = \pi(g_\eps)v \to \pi(g)v$ implies that $\pi(g)v = v$. 
\smartqed\qed\end{proof}

As an immediate consequence, we obtain: 

\begin{lemma} \mlabel{lem:p.4} Let $(\pi, \cH_\pi)$ be a continuous 
unitary representation of $\U(\cH)$ with $\U_\infty(\cH)\subeq \ker \pi$ 
and $j \in \Z$ with 
$\pi(\zeta \1) = \zeta^j \1$ for $\zeta \in \T$. 
If $\lambda$ is contained in the essential spectrum of $g$, 
then $\lambda^j$ is an eigenvalue of $\pi(g)$. 
\end{lemma}

\begin{proof} Lemma~\ref{lem:p.3} implies that 
$\pi(\lambda^{-1}g)$ has a non-zero fixed vector $v$, 
and this means that $\pi(g)v = \lambda^j v$. 
\smartqed\qed\end{proof}

\begin{theorem} \mlabel{thm:p.1b} If $\cH$ is a separable Hilbert space over 
$\K \in \{\R, \C, \H\}$, then every continuous unitary representation of 
$\U(\cH)/\U_{\infty}(\cH)_0$ on a separable Hilbert space is trivial. 
\end{theorem}

\begin{proof} (a) We start with the case $\K = \C$. 
Let $(\pi, \cH_\pi)$ be a separable continuous unitary representation 
of the Banach--Lie group $\U(\cH)$ with $\U_\infty(\cH) \subeq \ker \pi$. 

{\bf Step $1$:} $\T \1 \subeq \ker \pi$: Decomposing 
the representation of the compact central subgroup 
$\T\1$, we may  w.l.o.g.\ assume that $\pi(\zeta\1) = \zeta^j\1$ 
for some $j \in \Z$. Let $g \in \U(\cH)$ be an element with 
uncountable essential spectrum. 
If $j \not=0$, then Lemma~\ref{lem:p.4} implies that 
$\pi(g)$ has uncountably many eigenvalues, which is impossible 
if $\cH_\pi$ is separable. Therefore $j = 0$, and this means that 
$\T \1\subeq \ker \pi$. 

{\bf Step $2$:} Let $P \in B(\cH)$ be an orthogonal projection 
with infinite rank. Then $\U(P\cH) \cong \U(\cH)$,
so that Step $1$ implies that 
$\T P + (\1 -P) \subeq \ker \pi.$ 
If $P$ has finite rank, then 
\[ \T P + (\1 -P) \subeq \U_\infty(\cH) \subeq \ker \pi.\] 
This implies that $\ker \pi$ contains all elements $g$ 
with $\Spec(g) \subeq \{ 1,\zeta\}$ for some $\zeta \in \T$. 
Since every element with finite spectrum is a finite product 
of such elements, it is also contained in $\ker \pi$. 
Finally we derive from the Spectral Theorem that the subset 
of elements with finite spectrum is dense in $\U(\cH)$, 
so that $\pi$ is trivial.\begin{footnote}{This argument simplifies 
Pickrell's argument that was based on the simplicity of the topological 
group $\U(\cH)/\T \U_\infty(\cH)$ (\cite{Ka52}).}\end{footnote}

(b) Next we consider the orthogonal group 
$\U_\R(\cH) = \OO(\cH)$ of a real Hilbert space~$\cH$. 
Since $\cH$ is infinite-dimensional, there exists an orthogonal 
complex structure $I \in \OO(\cH)$. Then 
$\tau(g) := IgI^{-1}$ defines an involution on 
$\OO(\cH)$ whose fixed point set is the unitary group 
$\U(\cH,I)$ of the complex Hilbert space $(\cH,I)$. 

Let $\pi \: \OO(\cH) \to \U(\cH_\pi)$ be a continuous separable 
unitary representation with 
$\SO_\infty(\cH) := \OO_\infty(\cH)_0 \subeq N := \ker \pi$. 
Applying (a) to $\pi\res_{\U(\cH,I)}$, 
it follows that $\U(\cH,I) \subeq N$ and hence in particular 
that $I \in N$. 

For $X^\top = -X$ and $\tau(X) = -X$ we then obtain 
\[ N \ni  I\exp(X) I^{-1} \exp(-X) 
= \exp(-X)\exp(-X) = \exp(-2X).\] 
This implies that $\L(N) = \{ X \in \fo(\cH) \: \exp(\R X) \subeq N\}= \fo(\cH)$, and since $\OO(\cH)$ 
is connected by Example~\ref{ex:1.2}, it follows that $N = \OO(\cH)$, i.e., 
that $\pi$ is trivial. 

(c) Now let $\cH$ be a quaternionic Hilbert space, considered as a 
right $\bH$-module. Realizing $\cH$ as $\ell^2(S,\bH)$ for some set $S$, 
we see that $\cK := \ell^2(S,\C)$ is a 
complex Hilbert space whose complex structure is 
given by left multiplication $\lambda_\cI$ with the basis element 
$\cI \in \bH$ (this map is $\bH$-linear) and 
we have a direct sum $\cH^\C = \cK \oplus \cK \cJ$ 
of complex Hilbert spaces. 

Let $\sigma \: \ell^2(S,\bH) \to \ell^2(S,\bH)$ be the real linear 
isometry given by $\sigma(v) = \cI v \cI^{-1}$ pointwise on $S$, so that 
$\cH^\sigma = \ell^2(S,\C) = \cK$ and 
$\cK \cJ = \cH^{-\sigma}$. 
Then $\tau(g) := \sigma g\sigma$ defines an involution 
on $\U_\bH(\cH)$ whose group of fixed points is isomorphic 
to the unitary group $\U(\cK)$ of the complex Hilbert space $\cK$, 
on which the complex structure is given by right multiplication 
with~$\cI$, which actually coincides with the left multiplication. 

Let $\pi \: \U_\bH(\cH) \to \U(\cH_\pi)$ be a continuous separable 
unitary representation with 
$\U_{\bH,\infty}(\cH) \subeq N := \ker \pi$. 
Applying (a) to $\pi\res_{\U(\cK)}$, 
it follows that $\U(\cK) \subeq N$ and hence in particular 
that $\lambda_\cI \in N$. 
On the Lie algebra level, $\fu(\cK)$ is complemented 
by
\[ \{ X \in \fu_\bH(\cH) \: \sigma X = -X \sigma\} 
= \{ X \in \fu_\bH(\cH) \: \lambda_\cI X = - X \lambda_\cI\}, \] 
and for any element of this space we have 
\[ N \ni  \lambda_\cI\exp(X) \lambda_\cI^{-1} \exp(-X) 
= \exp(-X)\exp(-X) = \exp(-2X).\] 
This implies that $\L(N) = \fu_\bH(\cH)$, and since $\U_\bH(\cH)$ 
is connected by Proposition~\ref{prop:1.3}(i), $N = \U_\bH(\cH)$, so 
that $\pi$ is trivial. 
\smartqed\qed\end{proof}

\begin{remark} For $\K = \R$, the group 
$\OO(\cH)/\SO_\infty(\cH)$ is the $2$-fold simply connected cover of 
the group $\OO(\cH)/\OO_\infty(\cH)$. Therefore the triviality of all separable 
representations of $\OO(\cH)/\OO_\infty(\cH)$ follows from the triviality 
of all separable representations of 
$\OO(\cH)/\SO_\infty(\cH) = \OO(\cH)/\OO_\infty(\cH)_0$. 
\end{remark}

\begin{problem} If $\cH$ is an inseparable Hilbert space, 
then we think that all separable unitary representations $(\pi, \cH)$ of 
$\U(\cH)$ should be trivial, but we can only show that $\ker \pi$ 
contains all operators for which $(g - \1)\cH$ is separable, i.e., all 
groups $\U(\cH_0)$, where $\cH_0 \subeq \cH$ is a separable subspace. 

The argument works as follows. From Remark~\ref{rem:insep} we know that all irreducible 
representations of $\U_\infty(\cH)$ are inseparable. 
Theorem~\ref{thm:insep} implies that $\U_\infty(\cH) \subeq \ker \pi$. 
Now Theorem~\ref{thm:p.1b} implies that $\ker \pi$ contains all subgroups 
$\U(\cH_0)$, where $\cH_0$ is a separable Hilbert space, and this proves our claim. 
\end{problem}

\subsection{Separable representations of the Lie group $\U(\cH)$} 

Based on Pickrell's Theorem on the triviality of the separable representations 
of the quotient Lie groups $\U(\cH)/\U_\infty(\cH)_0$, we can now 
determined all separable continuous unitary representations 
of the full unitary group $\U(\cH)$. 

\begin{theorem}\mlabel{thm:p.2} Let $\cH$ be a separable $\K$-Hilbert space. 
Then every separable 
continuous unitary representation $(\pi, \cH_\pi)$ of the Banach--Lie group $\U(\cH)$ 
has the following properties: 
\begin{description}
\item[\rm(i)] It is continuous with respect to the strong operator topology 
  on $\U(\cH)$.
\item[\rm(ii)] Its restriction to $\U_\infty(\cH)_0$ has the same commutant. 
\item[\rm(iii)] It is a direct sum of bounded irreducible representations. 
\item[\rm(iv)] Every irreducible separable representation is of the form 
\[ \begin{cases}
\bS_\lambda(\cH_\C) \subeq (\cH_\C)^{\otimes N}, 
\lambda \in \Part(N), & \text{ for } \K = \R,  \\ 
\bS_\lambda(\cH) \otimes 
\bS_\mu(\oline\cH) \subeq 
\cH^{\otimes N} \otimes \oline\cH^{\otimes M}, 
\lambda \in \Part(N), \mu \in \Part(M),  & \text{ for } \K = \C,  \\
\bS_\lambda(\cH^\C) \subeq (\cH^\C)^{\otimes N}, 
\lambda \in \Part(N),  & \text { for } \K = \H.  \\
\end{cases} \]
\item[\rm(v)] $\pi$ extends uniquely to a strongly continuous 
representation of the overgroup $\U(\cH)^\sharp$ with the same commutant. 
\end{description}
\end{theorem}

\begin{proof} (i) From Theorem~\ref{thm:class}(c)   
we know that $\pi_\infty := \pi\res_{\U_\infty(\cH)}$ 
extends to a unique continuous unitary representation 
$\oline\pi$ of $\U(\cH)_s$ on $\cH_\pi$. 
In particular, the action of 
$\U(\cH)$ on the unitary dual of the normal subgroup 
$\U_\infty(\cH)_0$ is trivial. Hence all isotypic subspaces $\cH_{[\lambda]}$ for 
$\pi_\infty$ are invariant under $\pi$. We may therefore assume that 
$\pi_\infty$ is isotypic, i.e., of the form 
$\1 \otimes \rho_\lambda$, where $(\rho_\lambda,V_\lambda)$ is an irreducible 
representation of $\U_\infty(\cH)_0$ (cf.\ Theorem~\ref{thm:class}).
Then $\oline\pi := \1 \otimes \oline\rho_\lambda$ is continuous with respect to the operator 
norm on $\U(\cH)$ because the representations of $\U(\cH)$ on the spaces 
$\cH_\C^{\otimes N}$ are norm-continuous (Theorem~\ref{thm:class}(c)). 

Now 
\[ \beta(g) := \pi(g)\oline\pi(g)^{-1} \in \pi(\U_\infty(\cH))'  = \oline\pi(\U(\cH))'\]
implies that $\beta \:  \U(\cH) \to \U(\cH_\pi)$ defines a separable norm-continuous 
unitary representation vanishing on $\U_\infty(\cH)$. By Theorem~\ref{thm:p.1b} it is trivial, 
so that $\pi = \oline\pi$. 

(ii) follows from (i) and the density of $\U_\infty(\cH)_0$ in $\U(\cH)_s$. 

(iii), (iv) now follow from Theorem~\ref{thm:class}.  

(v) In view of (iii), assertion (v) 
reduces to the case of irreducible representations. 
In this case (v) follows from the concrete classification (iv) and the 
description of the overgroups $\U(\cH)^\sharp$ 
in Lemma~\ref{lem:k-grp}. 
\smartqed\qed\end{proof}

\begin{corollary} \mlabel{cor:e.2} Let $K$ be a quotient of a product 
$K_1 \times \cdots \times K_n$, where each $K_j$ 
is compact, a quotient of 
some group $\U(\cH)$ or $\U_\infty(\cH)_0$, where $\cH$ is a separable $\K$-Hilbert space. 
Then every separable continuous unitary representation 
$\pi$ of $K$ is a direct sum of irreducible representations 
which are bounded. 
\end{corollary}

The preceding corollary means that the separable representation 
theory of $K$ is very similar to the  representation theory of a compact 
group.

\subsection{Classification of irreducible representations by highest weights} 

We choose an orthonormal basis $(e_j)_{j \in J}$ in the complex Hilbert space 
$\cH$ and write $T \cong \T^J$ for the corresponding group of diagonal matrices. 
Characters of this group correspond to finitely supported functions 
$\lambda \: J \to \Z$ via $\chi_\lambda(t) = \prod_{j \in J} t_j^{\lambda_j}$. 
For the subgroup $T(\infty)$ of those diagonal matrices $t$ for which 
$t - \1$ has finite rank, any function $\lambda\: J \to \Z$ defines a 
character. Accordingly, each 
$\lambda = (\lambda_j)_{j \in J} \in \Z^J$ defines a uniquely determined 
{\it unitary highest weight representation} $(\pi_\lambda, \cH_\lambda)$ 
of $\U(\infty,\C)$ (\cite{Ne04, Ne98}). This representation is uniquely determined 
by the property that its weight set with respect to the diagonal 
subgroup $T \cong \T^{(J)}$, whose character group is $\Z^J$, coincides with 
\[ \conv(\cW\lambda) \cap (\lambda + \cQ), \quad \mbox{ where } \quad 
\cQ \subeq \hat T \] 
is the root group and $\cW$ is the group of finite permutations of the set~$J$. 

\begin{proposition} A unitary highest weight representation  
$(\pi_\lambda,\cH_\lambda)$ of $\U(\infty,\C)$ is tame 
if and only if $\lambda \: \N \to \Z$ is finitely supported. 
\end{proposition}

\begin{proof} If $\pi_\lambda$ is a tame
representation, then its restriction to the diagonal subgroup is
tame. Since this representation is diagonalizable, this means that
each weight has finite support. It follows in particular that
$\lambda$ has finite support. 

If, conversely, $\lambda$ has finite support, then 
we write $\lambda = \lambda_+ - \lambda_-$, where 
$\lambda_\pm$ are non-negative with finite disjoint support. 
Then $\cH_\lambda$ can be embedded into \break 
$\bS_{\lambda_+}(\cH) \otimes \bS_{\lambda_-}(\oline\cH)$ (\cite{Ne98}), hence it is tame. 
\smartqed\qed\end{proof}

\begin{example} {\bf $\K =\R$}:
In the infinite-dimensional 
real Hilbert space $\cH$ we fix a complex structure $I$. 
Then there exists a real orthonormal basis of the form 
${\{ e_j, Ie_j \: j \in J\}}$. 
Then the subgroup $T \subeq \OO(\cH)$ preserving all the planes 
$\R e_j + I \R e_j$ is maximal abelian. 
In $\cH_\C$ the elements $e_j^\pm := \frac{1}{\sqrt 2}(e_j \mp I e_j)$ 
form an orthonormal basis, and we write 
$2J := J \times \{\pm\}$ for the corresponding index set. In 
$\OO(\cH)^\sharp = \U(\cH_\C)$, the corresponding diagonal subgroup 
$T^\sharp \cong \T^{2J}$ is maximal abelian. The corresponding maximal 
torus $T_\C$ of $\OO(\cH)_\C \subeq \GL(\cH_\C)$ corresponds to diagonal 
matrices $d$ acting by $d e_j^\pm = d_j^{\pm 1} e_j^\pm$. 

For a character $\chi_\mu$ of $T^\sharp$ with 
$\mu \: 2J \to \Z$, the corresponding character of 
$T$ is given by the finitely supported function 
$\mu^\flat \: J \to \Z$ with $\mu^\flat_j = \mu_{j,+} - \mu_{j,-}$. 
If $\lambda \: 2J \to \N_0$ has finite support, then the corresponding 
irreducible representation of $\U(\cH_\C)$ occurs as 
some $\bS_\lambda(\cH_\C)$ in $\cH_\C^{\otimes N}$, where $\sum_{j \in 2J} \lambda_j = N$. 
From the Classification Theorem~\ref{thm:class} it follows that the restriction 
of $\pi_\lambda$ to $\OO(\cH)$ is irreducible. 
The corresponding highest weight is $\lambda^\flat$. 
On the level of highest weights, it is clear that,  
for each finitely supported weight $\lambda \: J \to \N_0$, 
we obtain by 
\[ \lambda^\sharp_{j,+} := \lambda_j \quad \mbox{ and } \quad 
\lambda^\sharp_{j,-} := 0\] 
a highest weight $\lambda^\sharp$ with $(\lambda^\sharp)^\flat= \lambda$. 
The irreducible representations of $\OO(\cH)$ are classified by 
orbits of the Weyl group $\cW$ in the set of finitely supported integral weights 
$\lambda \: J \to \Z$ of the root system $D_{2J}$ (cf.\ \cite[Sect.~VII]{Ne98}). 
Each orbit has a non-negative representative, and then 
$\lambda^\sharp$ is the highest weight of the corresponding 
representation $\pi_{\lambda^\sharp}$ of $\U(\cH_\C)$. 
\end{example}

\begin{example} {\bf $\K = \C$}: 
Let $(e_j)_{j\in J}$ be an ONB of $\cH$. 
In $\U(\cH)^\sharp \cong \U(\cH) \times \U(\oline\cH)$ 
we have the maximal abelian subgroup 
$T^\sharp = T \times T$, where $T \cong \T^J$ is the subgroup 
of diagonal matrices in $\U(\cH)$ with respect to the ONB $(e_j)_{j \in J}$. 

Let $2J := J \times \{\pm \}$, so that $T^\sharp \cong \T^{2J}$. 
For a finitely supported function $\mu \: 2J \to \Z$, the corresponding character of 
$T$ is given by 
$\mu^\flat \: J \to \Z$, defined by $\mu^\flat_j = \mu_{j,+} - \mu_{j,-}$. 
If $\lambda \: 2J \to \N_0$ has finite support, 
and $\lambda= \lambda_+ - \lambda_-$ with non-negative summands 
$\lambda_\pm$ supported in $J \times \{\pm\}$, respectively, the corresponding 
irreducible representation $\pi_\lambda$ lives on 
$\bS_{\lambda_+}(\cH) \otimes \bS_{\lambda_-}(\oline\cH) 
\subeq \cH^{\otimes N} \otimes \oline\cH^{\otimes M}$, where 
$N = \sum_{\lambda_j > 0} \lambda_j$ and 
$M = -\sum_{\lambda_j < 0} \lambda_j$. 
From the Classification Theorem~\ref{thm:class} it follows that the restriction 
of $\pi_\lambda$ to $\U(\cH)$ is irreducible. 
The corresponding highest weight is $\lambda^\flat = \lambda_+ - \lambda_-$. 
For each finitely supported weight $\lambda = \lambda_+ - \lambda_-\: J \to \N_0$, 
we obtain by 
$\lambda^\sharp_{j,\pm} := \lambda_{\pm,j}, j \in J$, 
a highest weight $\lambda^\sharp$ with $(\lambda^\sharp)^\flat= \lambda$. 
The irreducible representations of $\U(\cH)$ are classified by 
orbits of the Weyl group $\cW\cong S_{(J)}$ in the set of finitely supported integral weights 
$\lambda \: J \to \Z$ of the root system $A_{J}$ (cf.\ \cite[Sect.~VII]{Ne98}). 
\end{example}

\begin{example} {\bf $\K = \H$}: 
In the quaternionic Hilbert space $\cH$ we consider the complex structure defined by 
multiplication with $\cI$, which leads to the complex Hilbert space $\cH^\C$. 
Then there exists a complex orthonormal basis of the form 
$\{ e_j, \cJ e_j \: j \in J\}$. 
We write $T^\sharp \subeq \U(\cH^\C)$ for the corresponding diagonal subgroup. 
Note that $T^\sharp \cong \T^{2J}$ for $2J := J \times \{\pm\}$. 
The subgroup $T := T^\sharp \cap \U(\cH) = (T^\sharp)^\cJ$ acts on 
the basis elements $e_{j,+} := e_j$ and $e_{j,-} := \cJ e_j$ by 
$d e_{j,\pm} = d_j^{\pm} e_{j,\pm}$. 

The classification of the irreducible representations by Weyl group orbits 
of finitely supported functions $\lambda \: J \to \Z$ (weights for the root 
system $B_J$) and their corresponding 
weights $\lambda^\sharp \: 2J \to \Z$ is completely analogous to the situation 
for $\K = \R$. The irreducible representation of $\U(\cH^\C)$ 
corresponding to $\lambda^\sharp$ is 
$\bS_{\lambda^\sharp}(\cH^\C)$. 
\end{example}

\begin{remark} (Segal's physical representations) 
In \cite{Se57} Segal studied unitary
representations of the full group $\U({\cal H})$,  
called {\it physical representations}. They are characterized by the
condition that their differential maps finite rank hermitian
projections to positive operators. Segal shows that 
physical representations decompose discretely into irreducible
physical representations which are precisely those 
occurring in the decomposition of finite tensor products $\cH^{\otimes N}$, $N \in\N_0$. 
In view of Pickrell's Theorem, this also follows from our classification 
of the separable representations of $\U(\cH)$. Since Segal's arguments 
never use the seprability of $\cH$, the corresponding result remains true 
for inseparable spaces as well. 
\end{remark}

\begin{problem} Theorem~\ref{thm:class} implies in particular 
that all continuous unitary representations of $K = \U_\infty(\cH)_0$ 
have a canonical extension to their overgroups $K^\sharp$ with the same commutant. 
The classification in terms of highest weights further implies that the 
representations of $K^\sharp$ obtained from this extension process are 
precisely those with non-negative weights. 

Conversely, it follows that all unitary representations of 
$K^\sharp$ with non-negative weights remain irreducible when restricted to $K$. 

One may ask a similar question for the smaller group 
$K^\sharp(\infty) \subeq K^\sharp$ or its completion with respect to the 
trace norm. Is it true that, for any unitary representation $\pi$ of 
$K^\sharp(\infty)$ whose weights on the diagonal subgroup are 
non-negative, $\pi(\U(\infty,\K))$ has the same commutant? 
As we explain below, this is not true. 

For the special case $\K = \R$ and $\lambda = \lambda_+ - \lambda_-$ finitely supported, 
the restriction of the representation $\pi_\lambda 
= \pi_{\lambda_+} \otimes \pi_{\lambda_-}^*$ 
of $K^\sharp(\infty)$ on $S_{\lambda_+}(\cH_\C) \otimes S_{\lambda_-}(\oline{\cH_\C})$ 
to the subgroup $K(\infty) = \SO(\infty,\R)$ is equivalent to the representation 
$\pi_{\lambda_+} \otimes \pi_{\lambda_-}$, which decomposes according 
to the standard Schur--Weyl theory. In particular, we obtain non-irreducible 
representations if $\lambda$ takes positive and negative values on $K(\infty)$. 
That this cannot be repaired by the positivity requirement on the 
weights of $K^\sharp(\infty)$ follows from the fact that the 
determinant $\det \: K^\sharp(\infty) \to \T$ restricts to the trivial character 
of $K(\infty)$, but tensoring with a power of $\det$, any bounded weight 
$\lambda$ can be made positive. 

Is it possible to characterize those irreducible highest weight 
representations $\pi_\lambda$ of 
$K^\sharp(\infty)$ whose restriction to $K(\infty)$ is irreducible? 
\end{problem}

\section{Non-existence of separable unitary representations 
for full operator groups} 
\mlabel{sec:5}

In this section we describe some consequences of the main results 
from \cite{Pi90}. We start with the description of $10$ symmetric pairs 
$(G,K)$ of groups of operators, where $G$ does not consist 
of unitary operators and $K \subeq G$ is ``maximal unitary''. They are 
infinite-dimensional analogs of certain non-compact real reductive Lie groups.  
The dual symmetric pairs $(G^c,K)$ have the property that $G^c$ consists of unitary operators, 
hence they are analogs of certain compact matrix groups. 

One of the main result of this section is that all 
separable unitary representations of the groups $G$ 
are trivial, but there are various refinements 
concerning restricted groups. 

\subsection{The $10$ symmetric pairs} 

Below we use the following notational conventions. We write 
$\OO(n) := \OO(n,\R)$, $\U(n) := \U(n,\C)$ and 
$\Sp(n) := \U(n,\H)$ for $n \in \N \cup \{\infty\}$. 
For a group $G$, we write 
$\Delta_G := \{ (g,g) \: g \in G \}$ 
for the diagonal subgroup of $G \times G$. 

If $\cH$ is a complex Hilbert space, 
then we write $I \in B(\cH_\C)$ for the $\C$-linear extension of the complex 
structure $Iv = iv$ on $\cH$. Then $D := -iI$ is a unitary involution that 
leads to the pseudo-unitary group $\U(\cH_\C, D) 
= \{ g \in \GL(\cH_\C) \: Dg^* D^{-1} = g^{-1}\}$ preserving the indefinite 
hermitian form $\la Dv,w\ra$. 
For the isometry group of the indefinite form 
$h((v_1, v_2), (w_1, w_2)) := \la v_1, w_1 \ra - \la v_2, w_2\ra$ on 
$\cH \times \cH$, we write $\U(\cH,\cH)$. 
Now the group 
\[ \OO^*(\cH_\C) := \U(\cH_\C,D) \cap \OO(\cH)_\C \] 
is a Lie group. Its Lie algebra 
$\fo^*(\cH_\C)$ satisfies 
$\fo^*(\cH_\C) \cap \fu(\cH_\C) \cong \fu(\cH)$ and it is 
a real form of $\fo(\cH)_\C$. It is easy to see that 
the symmetric pair $(\OO^*(\cH_\C),\U(\cH))$ 
is dual to $(\OO(\cH^\R),\U(\cH))$. \\

\pagebreak 
{{\bf Non-unitary symmetric pairs} \\[3mm]
\begin{tabular}{|l|l |l |l|l|}
       \hline
& non-unit.\ locally finite\ $(G(\infty), K(\infty))$ 
& operator group $(G,K)$ & $K$ & $\K$ \\
   \hline\hline
1 &  $(\GL(\infty,\C), \U(\infty))$ & $(\GL(\cH), \U(\cH))$ & $\U(\cH)$ & $\C$ \\ 
 \hline
2 & $(\SO(\infty,\C), \SO(\infty))$ & $(\OO(\cH)_\C, \OO(\cH))$ & $\OO(\cH)$ & $\R$ \\ 
 \hline
3 & $(\Sp(\infty,\C), \Sp(\infty))$& $(\U_\H(\cH)_\C, \U_\H(\cH))$ & $\U_\H(\cH)$ & $\H$ \\ 
 \hline
4& $(\U(\infty, \infty), \U(\infty)^2)$& $(\U(\cH,\cH), \U(\cH)^2)$ & $\U(\cH)^2$ 
& $\C$ \\  \hline
5 & $(\SO(\infty,\infty), \SO(\infty)^2)$ & 
$(\OO(\cH,\cH), \OO(\cH)^2)$ & $\OO(\cH)^2$ & $\R$ \\  \hline
6  & $(\Sp(\infty,\infty), \Sp(\infty)^2)$& 
$(\U_\H(\cH,\cH), \U_\H(\cH)^2)$ & $\U_\H(\cH)^2$ & $\H$ \\  \hline
7 & $(\Sp(2\infty,\R), \U(\infty))$& 
$(\Sp(\cH), \U(\cH))$ & $\U(\cH)$ & $\C$ \\  \hline
8 & $(\SO(2\infty), \U(\infty))$& 
$(\OO^*(\cH_\C), \U(\cH))$ & $\U(\cH)$ & $\C$ \\  \hline
9 & $(\GL(\infty,\R), \OO(\infty))$& 
$(\GL(\cH), \OO(\cH))$ & $\OO(\cH)$ & $\R$ \\  \hline
10 & $(\GL(\infty,\H), \Sp(\infty))$& 
$(\GL_\H(\cH), \U_\H(\cH))$ & $\U_\H(\cH)$ & $\H$ \\  \hline
\end{tabular}}\\[3mm] 

{{\bf Unitary symmetric pairs} \\[3mm]
\begin{tabular}{|l|l |l |l|l|}
       \hline
& unitary locally finite\ $(G^c(\infty), K(\infty))$ 
& unitary operator group $(G^c,K)$ & $K$ & $\K$ \\
   \hline\hline
1 & $(\U(\infty)^2, \Delta_{\U(\infty)})$& $(\U(\cH)^2, \Delta_{\U(\cH)})$
 & $\U(\cH)$  & $\C$ \\  \hline
2 & $(\SO(\infty)^2, \Delta_{\SO(\infty)})$& $(\OO(\cH)^2, \Delta_{\OO(\cH)})$
& $\OO(\cH)$  &$\R$ \\  \hline
3 & $(\Sp(\infty)^2, \Delta_{\Sp(\infty)})$& $(\U_\H(\cH)^2, \Delta_{\U_\H(\cH)})$
& $\U_\H(\cH)$ & $\H$ \\  \hline
4& $(\U(2\infty), \U(\infty)^2)$& $(\U(\cH \oplus \cH), \U(\cH)^2)$ 
& $\U(\cH)^2$  & $\C$ \\  \hline
5 & $(\SO(2\infty), \SO(\infty)^2)$ & 
$(\OO(\cH \oplus \cH), \OO(\cH)^2)$ & $\OO(\cH)^2$ & $\R$ \\  \hline
6  & $(\Sp(2\infty), \Sp(\infty)^2)$& 
$(\U_\H(\cH \oplus \cH), \U_\H(\cH)^2)$ & $\U_\H(\cH)^2$ & $\H$\\  \hline
7 & $(\Sp(\infty), \U(\infty))$& 
$(\U_\H(\cH\otimes_\C \H), \U(\cH))$ & $\U(\cH)$ & $\C$ \\  \hline
8 & $(\SO(2\infty), \U(\infty))$& 
$(\OO(\cH^\R), \U(\cH))$ & $\U(\cH)$ & $\C$ \\  \hline
9 & $(\U(\infty), \OO(\infty))$& 
$(\U(\cH_\C), \OO(\cH))$ & $\OO(\cH)$ & $\R$ \\  \hline
10 & $(\U(2\infty), \Sp(\infty))$& 
$(\U(\cH^\C), \U_\H(\cH))$ & $\U_\H(\cH)$ & $\H$\\  \hline
\end{tabular}}\\

\begin{remark} (a) The unitary symmetric pairs $(1)$-$(3)$ are 
of group type and their non-unitary duals are complex groups. 

(b) The non-unitary pairs $(4)$-$(6)$ are the symmetric pairs associated
to pseudo-unitary groups of indefinite hermitian forms $\beta$ 
with the matrix 
$D = \pmat{\1 & 0 \\ 0 & -\1}$ on $\cH^2$. Accordingly, the corresponding 
symmetric spaces can be considered as Gra\ss{}mannians of ``maximal 
positive subspaces'' for $\beta$. 

(c) The symmetric spaces corresponding to (7) and (8) are spaces 
of complex structures on real spaces. The space
$\Sp(\cH)/\U(\cH)$ is the space of positive symplectic complex structures on the 
real symplectic spaces $(\cH,\omega)$, where 
$\omega(v,w) = \Im  \la v,w\ra$. Likewise 
$\OO(\cH^\R)/\U(\cH)$ is the space of orthogonal complex structures 
on the real Hilbert space $\cH^\R$. 

(d) The spaces (4), (7) and (8) are of hermitian type 
(cf.\ \cite{Ne12}). 

(e) The spaces (1), (9) and (10) are those occurring naturally 
for overgroups of unitary groups (cf.\ Example~\ref{ex:2.3}). 
\end{remark}

\subsection{Restricted symmetric pairs} 

For each symmetric pair $(G,K)$ of non-unitary type and 
$1 \leq q \leq \infty$, we obtain a {\it restricted symmetric pair} 
$(G_{(q)},K)$, defined by 
\[ G_{(q)} := \{ g \in G \: \tr(|g^*g - \1|^q) <  \infty \}.\] 
If $\g = \fk \oplus \fp$ with $\fp = \{ X \in \g \: X^* = X\}$, then the 
Lie algebra of $G^c$ is $\g_{(q)} = \fk \oplus \fp_{(q)}$, where 
$\fp_{(q)} = \fp \cap B_q(\cH)$. 
The corresponding dual symmetric pair is $(G^c_{(q)}, K)$ with 
$\g^c_{(q)}= \fk \oplus i \fp_{(q)}$. 
We also write 
\[ G_{\infty, (q)} := G_{(q)} \cap (\1 + K(\cH)) 
= K_\infty \exp(\fp_{(q)})\]  
for the closure of $G(\infty)$ in $G_{ (q)}$.

\begin{proposition} Spherical representations of 
any pair $(G(\infty), K(\infty))$ of unitary 
or non-unitary type are direct integrals of irreducible ones. 
\end{proposition}

\begin{proof} This is \cite[Prop.~2.4]{Pi90}, but it also follows from the general 
Theorem~\ref{thm:dirlim}~below.
\smartqed\qed\end{proof}

Combining the preceding proposition with two-sided estimates 
on the behavior of spherical functions near the identity, Pickrell proved: 

\begin{proposition} {\rm(\cite[Prop.~6.11]{Pi90})} Irreducible real 
spherical functions of the 
direct limit pairs $(G(\infty), K(\infty))$ always extend to spherical 
functions of $G_{(2)}$ 
and, for $q > 2$, all spherical functions on $G_{(q)}$ vanish. 
\end{proposition} 

If $v$ is a $C^1$-spherical vector for 
the unitary representation $(\pi, \cH_\pi)$ of $(G_{(q)},K)$, then 
$\beta(X,Y) := \la \dd\pi(X)v, \dd\pi(Y)v\ra$ 
defines a continuous $K$-invariant positive semidefinite symmetric 
bilinear form on $\fp_{(q)}$. Therefore one can also show that $v$ is fixed by 
the whole group $G_{(q)}$ by showing that $\beta = 0$ using the 
following lemma. 

\begin{lemma} \mlabel{lem:5.4} The following assertion holds for the $K$-action 
on $\fp_{(q)}$: 
\begin{description}
\item[\rm(i)] $[\fk, \fp] = \fp$. 
\item[\rm(ii)] $\fp_{(2)}$ is an irreducible representation. 
\item[\rm(iii)] For $q > 2$, 
every continuous $K$-invariant symmetric bilinear 
form on $\fp_{(q)}$ vanishes. 
\end{description}
\end{lemma}

\begin{proof} (i), (ii): 
We check these conditions for all $10$ families: 

{\bf (1)-(3)}  Then $\fp = i\fk$ with $\fk = \fu(\cH)$. 
Since $\fk$ is perfect by \cite[Lemma~I.3]{Ne02}, we obtain 
$[\fk,\fp] = i[\fk,\fk] = i \fk = \fp$. 

Here $\fp_{(2)} = i \fu_2(\cH)$, and since $\fu_2(\cH)$ is a simple 
Hilbert--Lie algebra (\cite{Sch60}), (ii) follows. 

{\bf (4)-(6)} In these cases $\g^c = \fu(\cH \oplus \cH)$, 
$\fk = \fu(\cH) \oplus \fu(\cH)$ and 
$\fp \cong \gl(\cH)$ with the $\fk$-module structure given by 
$(X,Y).Z := XZ-ZY$. Since $\fu(\cH)$ contains invertible elements 
$X_0$, and $(X_0,0).Z = X_0 Z$, it follows that 
$\fp = [\fk,\fp]$. 

Here $\fp_{(2)} \cong \gl_2(\cH)$ is the space of Hilbert--Schmidt operators on $\cH$. 
This immediately implies the irreducibility of the representation of 
$\K = \R, \C$. For $\K = \H$, we have 
$\fp_{(2),\C} \cong \gl_2(\cH^\C)$, and since the representation of 
$\U(\cH)$ on $\cH^\C$ is irreducible (we have 
$\U(\cH)_\C \cong \Sp(\cH^\C)$), (ii) follows. 

{\bf (7)-(8)} In these two cases the center $\fz := i\1$ of $\fk \cong \fu(\cH)$ 
satisfies $\fp = [\fz, \fp]$, which implies (i). 
The Lie algebra $\g_{(2)} = \fk \oplus \fp_{(2)}$ corresponds to the 
automorphism group of an irreducible hermitian symmetric space 
(cf.\ \cite[Thm.~2.6]{Ne12} and the subsequent discussion). 
This implies that the representation of $K$ on the complex 
Hilbert space $\fp_{(2)}$ is irreducible. 

{\bf (9)} Here $\g = \gl(\cH)$, $\fk = \fo(\cH)$ and $\fp = \Sym(\cH)$. 
For any complex structure $I \in\fo(\cH)$ we then obtain 
\[ \fp 
= \Herm(\cH,I) \oplus [I,\fp] 
= [\fu(\cH,I), \Herm(\cH,I)] \oplus [I,\fp] \subeq [\fk,\fp].\] 
This proves (i). 

Next we observe that $\fp_{(2)} = \Sym_2(\cH)$ satisfies 
$\fp_{(2),\C} \cong \Sym_2(\cH_\C) \cong S^2(\oline{\cH_\C})$, 
hence is irreducible by Theorem~\ref{thm:class}. 

{\bf (10)} Here $\g = \gl(\cH)$, $\fk = \fu(\cH)$ and $\fp = \Herm(\cH)$ for $\K =\H$. 
With the aid of an orthonormal basis, we find a real Hilbert space 
$\cK$ with $\cH \cong\cK \otimes_\R \bH$, where $\bH$ acts by right multiplication. 
This leads to an isomorphism $B_\H(\cH) \cong B_\R(\cK) \otimes_\R \H$ as real 
involutive algebras. In particular, 
\[ \Herm(\cH) \cong \Sym(\cK)\otimes \1 \oplus \Asym(\cK) \otimes \Aherm(\H). \]
Therefore (i) follows from $\Aherm(\H) = [\Aherm(\H), \Aherm(\H)]$ and 
from $\Sym(\cK) = [\fo(\cK), \Sym(\cK)]$, which we derive from (9). 

To verify (ii), we observe that $\fp_{(2)} = \Herm_2(\cH)$. 
From Kaup's classification of the real symmetric Cartan domains 
\cite{Ka97} it follows that $\fp_{(2)}$ is a real form of the 
complex $JH^*$-triple $\Skew(\cH^\C)$ of skew symmetric bilinear forms 
on $\cH^\C$, labelled by $II_{2n}^\H$. 
Since the action of $\U(\cH)$ on $\fp_{(2),\C} \cong \Skew(\cH^\C) \cong 
\Lambda^2(\oline{\cH_\C})$ extends 
to the overgroup $\U(\cH^\C)$ with the same commutant, 
the irreducibility of the resulting representation implies 
that the representation of $\U(\cH)$ on the real Hilbert space $\fp_{(2)}$ is 
irreducible as well. 

(iii) can be derived from (i). If $\beta \: \fp_{(q)} \times \fp_{(q)} \to \R$ 
is a continuous invariant symmetric bilinear form, then the same holds 
for its restriction to $\fp_{(2)}$. The simplicity of the representation 
on $\fp_{(2)}$ now implies that it is a multiple of the canonical 
form on $\fp_{(2)}$ given by the trace. But this form does not extend continuously 
to $\fp_{(q)}$ for any $q > 2$. 
\smartqed\qed\end{proof}

\begin{proposition} \mlabel{prop:q-res} 
{\rm(a)} Separable unitary representations of 
$G_{(q)}$, $q \geq 1$, are completely determined by their restrictions to $G(\infty)$.

\nin{\rm(b)} Conversely, every continuous separable unitary representation of 
$G_{\infty, (q)}, q \geq 1$ extends to a continuous unitary representation of $G_{(q)}$.  
\end{proposition}

\begin{proof} (cf.\ \cite[Prop.~5.1]{Pi90}) 
(a) Since $\fp(\infty)$ is dense in $\fp_{(q)}$, this follows from 
Theorem~\ref{thm:p.2}(i), applied to~$K$. 

(b) Since $K$ acts smoothly by conjugation on $G_{\infty,(q)}$, we can form the 
Lie group $G_{\infty,(q)} \rtimes K$ and note that the multiplication map 
to $G_{(q)}$ defines an isomorphism 
$(G_{\infty,(q)} \rtimes K)/K_\infty \to G_{(q)}$. 
Therefore the existence of the extension of $G_{(q)}$ follows 
from the uniqueness of the extension from $K_\infty$ to $K$ 
(Theorem~\ref{thm:class}(c)). 
\smartqed\qed\end{proof}

\begin{theorem} \mlabel{thm:a} 
If $(G,K)$ is one of the $10$ symmetric pairs  of non-unitary 
type, then, for $q > 2$, all separable projective 
unitary representations of 
$G_{(q)}$ and all projective unitary representations of $G_{\infty,(q)}$ are trivial. 
\end{theorem}

\begin{proof} If $\pi \: G_{(q)} \to \PU(\cH_\pi)$ is a continuous 
separable projective unitary representation, then composing with the 
conjugation representation of $\PU(\cH_\pi)$ on the Hilbert space 
$B_2(\cH_\pi)$ leads to a separable unitary representation of 
$G_{(q)}$ on $B_2(\cH_\pi)$. If we can show that this representation 
is trivial, then $\pi$ is trivial as well. Therefore it suffices 
to consider unitary representations. 

Since the group $G_{\infty,(q)}$ is separable, all its continuous unitary 
representations are direct sums of separable ones. 
Hence, in view of Proposition~\ref{prop:q-res}, the triviality of 
all continuous unitary representations of $G_{\infty, (q)}$ is equivalent to the 
triviality of all separable continuous unitary 
representations of $G_{(q)}$. We may therefore 
restrict our attention to separable representations of $G_{(q)}$.

Let $(\pi, \cH)$ be a continuous separable unitary representation 
of $G_{(q)}$. In view of Theorem~\ref{thm:p.2}, it is a direct sum of representations 
generated by the subspace $\cH^{K_n}$ for some $n \in \N$.  Any 
$v \in \cH^{K_n}$ generates a spherical subrepresentation of the subgroup
\[ G_{(q),n} := \{ g \in G_{(q)} \: g e_j = e_j, j =1,\ldots, n\}.\] 
Now Theorem~\ref{thm:a} implies that $v$ is fixed by $G_{(q),n}$. 

It remains to show that $G_{(q)}$ fixes $v$. 
In view of Proposition~\ref{prop:q-res}, it suffices to show that, for 
every $m > n$, $G(m)$ fixes $v$. The group 
$G(m)$ is reductive with maximal compact subgroup $K(m)$, and 
$G(m)_n$ is a non-compact subgroup. 

{\bf Case 1:} We first assume that the center of $G(m)_n$ is compact, which is the case 
for $G(m) \not= \GL(m,\K)$ (this excludes 1,8 and 9). Then $G(m)$ is minimal 
in the sense that every continuous bijection onto a topological group is 
open, and this property is inherited by all its quotient 
groups (\cite[Lemma~2.2]{Ma97}). In view of \cite[Prop.~3.4]{Ma97}, 
all matrix coefficients of irreducible unitary representations 
$(\rho,\cH_\rho)$ of quotients of $G(m)$ vanish at infinity 
of $G(m)/\ker \rho$. If $G(m)_n \not\subeq\ker \rho$, then the image of 
$G(m)_n$ in the quotient group is non-compact, so that the only vector in 
$\cH_\rho$ fixed by $G(m)_n$ is~$0$. Since every continuous unitary 
representation of $G(m)$ is a direct integral of irreducible ones, 
it follows that every $G(m)_n$-fixed vector in a unitary representation 
is fixed by $G(m)$. 

{\bf Case 2:} If $G(m) = \GL(m,\K)$, then $Z = \R^\times_+ \1$ is a non-compact 
subgroup of the center and the homomorphism 
$\chi \: G \to \R^\times_+, \chi(g) := |\det_\R(g)|$ is surjective. 
Therefore $S(m) := \ker \chi$ has compact center and satisfies 
$G(m) = Z S(m)$. 
The preceding argument now implies that every fixed vector 
for $S(m)_n$ in a unitary representation of $G(m)$ is fixed by $S(m)$. 
Since $\chi\res_{G(m)_n}$ is non-trivial, we conclude that 
every fixed vector for $G(m)_n$ in a unitary representation is fixed by $G(m)$. 

Combining both cases, we see that, in every unitary representation 
of $G(m)$, the subgroup 
$G(m)_n$ and $G(m)$ have the same fixed vectors, and this implies that 
every $G_{(q),n}$ fixed-vector is fixed by $G_{(q)}$. 
\smartqed\qed\end{proof}

\begin{theorem} {\rm(\cite[Prop.~7.1]{Pi90})} 
If $(G,K)$ is one of the $10$ symmetric pairs of unitary type, 
then, for $q > 2$, every separable continuous projective 
unitary representation of $G_{(q)}$ 
extends uniquely to a representation of $G$ that is continuous with respect to the strong 
operator topology on~$G$. In particular, it is a direct sum of irreducible ones 
which are determined by Theorem~\ref{thm:p.2}. 
\end{theorem}

\begin{proof} With similar arguments as in the preceding proof, we 
see that every separable unitary representation of $G_{(q)}$ is a direct 
sum of representations generated by the fixed point space 
of some subgroup $G_{(q),n}$. Therefore its restriction to 
$G(\infty)$ is tame, so that Theorem~\ref{thm:samerep} 
and Theorem~\ref{thm:class}  apply. 
\smartqed\qed\end{proof}

\begin{theorem} \mlabel{thm:b} If $(G,K)$ is one of the $10$ symmetric pairs of non-unitary 
type, then all separable unitary representations of $G$ are trivial. 
\end{theorem}

\begin{proof} Let $(\pi, \cH)$ be a continuous separable unitary 
representation of $G$. We know already from Theorem~\ref{thm:a} that 
$G_{(q)} \subeq N := \ker \pi$ holds for $q > 2$. 
Now $N$ is a closed normal subgroup containing 
$K$ and its Lie algebra therefore contains 
$[\fk,\fp] = \fp$ as well (Lemma~\ref{lem:5.4}). This proves that~$N = G$. 
\smartqed\qed\end{proof}

\appendix

\section{Positive definite functions} 
\mlabel{app:b}

In this appendix we recall some results and definitions concerning 
operator-valued positive definite functions. 

\begin{definition} {\rm Let $\cA$ be a $C^*$-algebra and $X$ be a set. 
A map $Q \: X \times X \to \cA$ is called a {\it positive definite kernel} 
if, for any finite sequence $(x_1,\ldots, x_n) \in X^n,$ the matrix 
$Q(x_i,x_j)_{i,j=1,\ldots, n} \in M(n,\cA)$ is a positive element.

For $\cA = B(V)$, $V$ a complex Hilbert space, this means that, for 
$v_1, \ldots, v_n \in V$, we always have 
$\sum_{i,j=1}^n \la Q(x_i,x_j)v_j, v_i \ra \geq 0$. }
\end{definition}

\begin{definition} {\rm Let $\cK$ be a Hilbert space, $G$ be a group, and 
$U \subeq G$ be a subset. A function $\phi \colon UU^{-1} \to B(\cK)$ 
is said to be {\it positive definite} if the kernel 
\[ Q_\phi \colon U \times U \to B(\cK), \quad (x,y) \mapsto \phi(xy^{-1})\] 
is positive definite. For $U = G$ we obtain the usual concept of a 
positive definite function on $G$.}
\end{definition}

\begin{remark} (Vector-valued GNS-construction) \mlabel{rem:gns} 
We briefly recall the bridge between positive definite 
functions and unitary representations. 

(a) If $(\pi, \cH)$ is a unitary representation of $G$, 
$V \subeq \cH$ a closed subspace and $P_V \: \cH \to V$ the 
orthogonal projection on $V$, then $\pi_V(g) := P_V \pi(g) P_V^*$ 
is a $B(V)$-valued positive definite function with 
$\pi_V(\1) = \1$. 

(b) If, conversely, $\phi \: G \to B(V)$ is positive definite with 
$\phi(\1) = \1$, then 
there exists a unique Hilbert subspace $\cH_\phi$ of the space $V^G$ 
of $V$-valued function on $G$ for which the evaluation maps 
$K_g \: \cH_\phi \to V, f \mapsto f(g)$ are continuous and satisfy 
$K_g K_h^* = \phi(gh^{-1})$ for $g,h \in G$ (\cite[Thm.~I.1.4]{Ne00}). 
Then right translation by elements of $G$ defines a unitary representation 
$(\pi_\phi(g)f)(x) = f(xg)$ on this space with $K_{xg}= K_x \circ \pi(g)$. 
It is called the {\it GNS-representation associated to $\rho$}.  
Now $K_1^* \: V \to \cH_\phi$ is an isometric embedding, so that 
we may identify $V$ with a closed subspace of $\cH_\phi$ 
and $K_\1$ with the orthogonal projection to~$V$. This leads to 
$\phi(g) = K_g K_1^* = K_\1 \pi(g) K_\1^*$, so that every 
positive definite function is of the form $\pi_V$. 
The construction also implies that $V \cong K_1^*(V)$ 
is $G$-cyclic in $\cH_\phi$. 
\end{remark}

For the following theorem, we simply note that all Banach--Lie groups 
are in particular Fr\'echet--BCH--Lie groups. 

\begin{theorem}  \mlabel{thm:extension} 
Let $G$ be a connected Fr\'echet--BCH--Lie group and 
$U \subeq G$ an open connected $\1$-neighborhood for which 
the natural homomorphism $\pi_1(U,\1) \to \pi_1(G)$ is surjective. 
If $\cK$ is  Hilbert space and 
$\phi \colon UU^{-1} \to B(\cK)$ an analytic positive definite function, 
then there exists a unique analytic positive definite 
function $\tilde\phi \colon G \to B(\cK)$ extending~$\phi$. 
\end{theorem}

\begin{proof} Let $q_G \: \tilde G \to G$ be the universal covering morphism. 
The assumption that $\pi_1(U) \to \pi_1(G)$ is surjective implies that 
$\tilde U := q_G^{-1}(U)$ is connected. Now 
$\tilde \phi := \phi \circ q_G \: \tilde U \tilde U^{-1} \to B(\cK)$ 
is an analytic positive definite function, hence extends by 
\cite[Thm.~A.7]{Ne12} to an analytic positive definite function 
$\tilde\phi$ on $\tilde G$. The restriction of $\tilde\phi$ to $\tilde U$ 
is constant on the fibers of $q_G$, which are of the form $g \ker(q_G)$. 
Using analyticity, we conclude that $\tilde \phi(gd) = \tilde\phi(g)$ holds 
for all $g \in \tilde G$ and $d \in \ker(q_G)$. Therefore $\tilde\phi$ factors 
through an analytic  function $\phi \: G \to B(U)$ which is obviously positive definite. 
\smartqed\qed\end{proof}

\begin{theorem} \mlabel{thm:locana} 
Let $G$ be a connected analytic  Fr\'echet--Lie group. 
Then a positive definite function 
$\phi \: G\to  B(V)$ which is analytic in an open identity neighborhood 
is analytic. 
\end{theorem}

\begin{proof} Since $\phi$ is positive definite, there exists a 
Hilbert space $\cH$ and a $Q \: G \to B(\cH,V)$ with 
$\phi(gh^{-1}) = Q_g Q_h^*$ for $g,h \in G$. 
Then the analyticity of the function $\phi$ in an open 
identity neighborhood of $G$ implies that the kernel 
$(g,h) \mapsto Q_g Q_h^*$ is analytic 
on a neighborhood of the diagonal $\Delta_G \subeq G \times G$. 
Therefore $Q$ is analytic by \cite[Thm.~A.3]{Ne12}, 
and this implies that $\phi(g) = Q_g Q_\1^*$ is analytic. 
\smartqed\qed\end{proof}

The following proposition describes a natural source 
of operator-valued positive definite functions. 

\begin{proposition} \mlabel{prop:b.1} Let $(\pi, \cH)$ be a unitary representation 
of the group $G$ and $H \subeq G$ be a subgroup. 
Let $V \subeq \cH$ be an isotypic $H$-subspace generating 
the $G$-module $\cH$ and $P_V \in B(\cH)$ be the orthogonal projection onto~$V$. 
Then $V$ is invariant under the commutant 
$\pi(G)' = B_G(\cH)$ and the map 
\[ \gamma \: B_G(\cH) \to B_H(V), \quad 
\gamma(A) = P_V A P_V \] 
is an injective morphism of von Neumann algebras 
whose range is the commutant of the image of the operator-valued 
positive definite function 
\[ \pi_V \: G \to B(V), \quad \pi_V(g) := P_V \pi(g) P_V.\] 
In particular, if the $H$-representation on $V$ is 
irreducible, then so is $\pi$. 
\end{proposition}

\begin{proof} That $\gamma$ is injective follows from the assumption that 
$V$ generates $\cH$ under $G$. 
If the representation 
$(\rho,V)$ of $H$ is irreducible, then $\im(\gamma) \subeq \C \1$ 
implies that $\pi(G)' = \C \1$, so that $\pi$ is irreducible. 

We now determine the range of $\gamma$. For any  $A \in B_G(\cH)$, we have 
\[ P_V \pi(g)P_V P_V A P_V = P_V \pi(g) A P_V 
= P_V A \pi(g) P_V = P_V A P_V P_V \pi(g) P_V, \] 
i.e., $\gamma(A) = P_V A P_V$ commutes with $\pi_V(G)$. 
Since $\gamma$ is a morphism of von Neumann algebras, 
its range is also a von Neumann algebra of $V$ commuting 
with $\pi_V(G)$. If, conversely, an orthogonal projection 
$Q = Q^* = Q^2 \in B_K(V)$ commutes with 
$\pi_V(G)$, then 
\[ P_V \pi(G) Q V 
= P_V \pi(G) P_V Q V 
= Q P_V \pi(G) P_V V \subeq Q V\] 
implies that the closed $G$-invariant subspace $\cH_Q \subeq \cH$ 
generated by $QV$ satisfies $P_V \cH_Q \subeq Q V$, and therefore 
$\cH_Q \cap V = QV$. 
For the orthogonal projection $\tilde Q \in B(\cH)$ onto $\cH_Q$, which 
is contained in $B_G(\cH)$, this means that $\tilde Q\res_V = Q$. 
This shows that $\im(\gamma) = \pi_V(G)'.$ 
\smartqed\qed\end{proof}

\begin{remark} The preceding proposition is particularly useful if we 
have specific information on the set $\pi_V(G)$. 
As $\pi_V(h_1 g h_2) = \rho(h_1) \pi_V(g) \rho(h_2)$, it is determined 
by the values of $\pi_V$ on representatives of the $H$-double cosets in~$G$. 

(a) In the context of the lowest $K$-type $(\rho,V)$ of a unitary highest 
weight representation (cf.\ \cite{Ne00}), we can expect that 
$\pi_V(G) \subeq \rho_\C(K_\C)$ (by Harish--Chandra decomposition), 
so that $\pi_V(G)' = \rho_\C(K_\C)' = \rho(K)'$ and 
$\gamma$ is surjective. 

(b) In the context of Section~\ref{sec:3} and \cite{Ol78}, the representation 
$(\rho,V)$ of $H$ extends to a representation 
$\tilde\rho$ of a semigroup $S \supeq H$ 
and we obtain $\pi_V(G)' = \tilde\rho(S)'$. 

In both situations we have a certain induction procedure 
from representations of $K$ and $S$, respectively, to $G$-representations 
which preserves the commutant but which need not be defined for 
every representation of $K$, resp., $S$. 
\end{remark}

\begin{lemma} {\rm(\cite{NO13})}
\mlabel{lem:mult} Let $(S,*)$ be a unital involutive semigroup 
and $\phi \: S \to B(\cF)$ be a positive definite function 
with $\phi(\1) = \1$. We write $(\pi_\phi, \cH_\phi)$ for the representation 
on the corresponding reproducing kernel Hilbert space 
$\cH_\phi \subeq \cF^S$ by $(\pi_\phi(s)f)(t) := f(ts)$. 
Then the inclusion 
\[ \iota \: \cF \to \cH_\phi, \quad \iota(v)(s) := \phi(s)v \]
 is surjective if and only if 
$\phi$ is multiplicative, i.e., a representation. 
\end{lemma}

\begin{remark} \mlabel{rem:b.6} 
The preceding lemma can also be expressed without referring to positive definite 
functions and the corresponding reproducing kernel space. In this context it asserts the following.
Let $\pi \: S \to B(\cH)$ be a $*$-representation of a unital involutive semigroup 
$(S,*)$, $\cF \subeq \cH$ a closed cyclic subspace and 
$P \: \cH \to \cF$ the orthogonal projection. Then the function 
\[ \phi \: S \to B(\cF), \quad \phi(s) := P \pi(s)P^* \] 
is multiplicative if and only if $\cF = \cH$. 
\end{remark}

\section{$C^*$-methods for  direct limit groups} 
\mlabel{app:c}

In this appendix we explain how to apply $C^*$-techniques to obtain 
direct integral decompositions of unitary representations 
of direct limit groups. 

We recall that, for a $C^*$-algebra $\cA$, its multiplier algebra 
$M(\cA)$ is a $C^*$-algebra containing $\cA$ as an ideal, and in 
every faithful representation $\cA \into B(\cH)$, it is given by 
\[ M(\cA) = \{M \in B(\cH) \: M\cA + \cA M \subeq \cA\}.\] 

Let $G = \indlim G_n$ be a direct limit of locally compact groups 
and $\alpha_n \: G_n \to G_{n+1}$ denote the connecting maps. We assume 
that these maps are closed embeddings.  
Then we have natural homomorphisms 
\[ \beta_n \: L^1(G_n) \to M(L^1(G_{n+1})) \] 
of Banach algebras, and since the action of $G_n$ on 
$L^1(G_{n+1})$ is continuous, $\beta_n$ is non-degenerate in the sense 
that $\beta(L^1(G_n)) \cdot L^1(G_{n+1})$ is dense in $L^1(G_{n+1})$. 
On the level of $C^*$-algebras we likewise obtain morphisms 
\[ \beta_n \: C^*(G_n) \to M(C^*(G_{n+1})). \] 

A state of $G$ (=normalized continuous positive definite function) 
now corresponds to a sequence $(\phi_n)$ of states of the groups $G_n$ 
with $\alpha_n^*\phi_{n+1} = \phi_n$ for every $n \in \N$. 
Passing to the $C^*$-algebras $C^*(G_n)$, we can view these functions 
also as states of the $C^*$-algebras. Then the compatibility condition 
is that the canonical extension $\tilde\phi_n$ of 
$\phi_n$ to the multiplier algebra satisfies 
\[ \beta_n^*\tilde\phi_{n+1}  = \phi_n.\]

\begin{remark} \mlabel{rem:5.16} 
The $\ell^1$-direct sum 
$\cL := \oplus^1 L^1(G_n)$ carries the structure of a Banach-$*$-algebra 
(cf.\ \cite{SV75}). 
Every unitary representation $(\pi, \cH)$ of $G$ defines a 
sequence of non-degenerate representations $\pi_n \: L^1(G_n) \to B(\cH_n)$ 
which are compatible in the sense that 
\[ \pi_n = \alpha_n^*\tilde\pi_{n+1}.\] 
Conversely, every such sequence of representations on a Hilbert space 
$\cH$ leads to a sequence 
$\rho_n \: G_n \to \U(\cH)$ of continuous unitary representations, which 
are uniquely determined by 
\[ \rho_n(g) = \tilde\pi_n(\eta_{G_n}(g)), \] 
where $\eta_{G_n} \: G_n \to M(L^1(G_n))$ denotes the canonical 
action by left multipliers. For 
$f \in L^1(G_{n})$ and $h \in L^1(G_{n+1})$ we then have 
\begin{align*}
& \rho_{n+1}(\alpha_n(g)) \pi_n(h) \pi_{n+1}(f) 
= \rho_{n+1}(\alpha_n(g)) \pi_{n+1}(\beta_n(h)f)
= \pi_{n+1}(\alpha_n(g)\beta_n(h)f)\\
&= \pi_{n+1}(\beta_n(g * h)f) 
= \pi_{n}(g * h)\pi_{n+1}(f) 
= \rho_n(g) \pi_n(h) \pi_{n+1}(f), 
\end{align*}
which leads to 
\[ \rho_{n+1} \circ \alpha_n = \rho_n.\] 
Therefore the sequence $(\rho_n)$ is coherent and thus defines a 
unitary representation of $G$ on $\cH$. We conclude that 
the continuous unitary representations of $G$ are in one-to-one 
correspondence with the coherent sequences of non-degenerate 
representations $(\pi_n)$ of the Banach-$*$-algebras 
$L^1(G_n)$ (cf.\ \cite[p.~60]{SV75}). 

Note that the non-degeneracy condition on the sequence $(\beta_n)$ 
is much stronger than the non-degeneracy condition on the corresponding 
representation of the algebra $\cL$. The group 
$G_{n+1}$ does not act by multipliers on $L^1(G_n)$, so that there 
is no multiplier action of $G$ on $\cL$. 
However, we have a sufficiently strong structure to apply 
$C^*$-techniques to unitary representations of $G$. 
\end{remark}

\begin{theorem} \mlabel{thm:4.1} Let $\cA$ be a separable $C^*$-algebra and 
$\pi \: \cA \to \cD$ a homomorphism into the algebra $\cD$ of decomposable 
operators on a direct integral $\cH$. Then there exists for each $x \in X$ a representation 
$(\pi_x, \cH_x)$ of $\cA$ such that $\pi \cong \int_X^{\oplus} \pi_x \, d\mu(x)$.  

If $\pi$ is non-degenerate and $\cH$ is separable, 
then almost all the representations $\pi_x$ are non-degenerate. 
\end{theorem}

\begin{proof} The first part is \cite[Lemma~8.3.1]{Dix64} (see also \cite{Ke78}). 
Suppose that $\pi$ is non-degenerate and let $(E_n)_{n \in \N}$ 
be an approximate identity on $\cA$.
Then $\pi(E_n) \to \1$ holds strongly in $\cH$ and 
\cite[Ch.\ II, no.~2.3, Prop.\ 4]{Dix69} implies the existence 
of a subsequence $(n_k)_{k \in \N}$ such that $\pi_x(E_{n_k}) \to \1$ holds strongly for almost 
every $x \in X$. For any such $x$, the representation $\pi_x$ is non-degenerate.   
\smartqed\qed\end{proof}

\begin{theorem} \mlabel{thm:dirlim} 
Let $G = \indlim G_n$ be a direct limit of separable locally compact 
groups with closed embeddings $G_n \into G_{n+1}$ and $(\pi, \cH)$ be a continuous separable unitary representation. 
For any maximal abelian subalgebra $\cA \subeq \pi(G)'$, we then obtain a 
direct integral decomposition $\pi \cong \int_X^\oplus \pi_x\, d\mu(x)$ into continuous 
unitary representations of $G$. 
\end{theorem}

\begin{proof} According to the classification of commutative $W^*$-algebras, we 
have $\cA \cong L^\infty(X,\mu)$ for a localizable measure space 
$(X,\fS,\mu)$ (\cite[Prop.~1.18.1]{Sa71}). We therefore obtain a 
direct integral decomposition of the corresponding Hilbert space $\cH$. 
To obtain a corresponding direct integral decomposition of the representation 
of $G$, we consider the $C^*$-algebra $\cB$ generated by 
the subalgebras $\cB_n$ generated by the image of the integrated representations 
$L^1(G_n) \to B(\cH)$. Then each $\cB_n$ is separable and therefore $\cB$ is also separable. 
Hence Theorem~\ref{thm:4.1} leads to non-degenerate representations 
$(\pi_x, \cH_x)$ of $\cB$ whose restriction to every $\cB_n$ is non-degenerate. 

In \cite{SV75}, the $\ell^1$-direct sum 
$\cL := \oplus_{n \in \N}^1 L^1(G_n)$ is used as a replacement for the group algebra. 
From the representation $\pi \: \cL \to \cB$ we obtain a representation 
$\rho_x$ of this Banach-$*$-algebra whose restrictions to the subalgebras $L^1(G_n)$ 
are non-degenerate. Now the argument in \cite[p.~60]{SV75} 
(see also Remark~\ref{rem:5.16} above) implies that the corresponding continuous 
unitary representations of the subgroups $G_n$ combine to a continuous 
unitary representation $(\rho_x, \cH_x)$ of $G$. 
\smartqed\qed\end{proof}

\begin{remark} Let $(\pi, \cH)$ be a continuous unitary representation 
of the direct limit $G= \indlim G_n$ of locally compact groups. 
Let $\cA_n := \pi_n(C^*(G_n))$ and write 
$\cA := \la \cA_n \: n \in \N\ra_{C^*}$ for the $C^*$-algebra 
generated by the $\cA_n$. Then $\cA'' = \pi(G)''$ follows immediately 
from $\cA_n'' = \pi_n(G_n)''$ for each $n$. 

From the the non-degeneracy of the multiplier action of 
$C^*(G_n)$ on $C^*(G_{n+1})$ it follows that 
\[ C^*(G_n) C^*(G_{n+1}) = C^*(G_{n+1}),\] 
which leads to 
\[ \cA_n \cA_{n+1}  = \cA_{n+1}.\] 
We have a decreasing sequence of closed-$*$-ideals 
\[ \cI_n := \oline{\sum_{k \geq n} \cA_k} \subeq \cA\]  
such that $G_n$ acts continuously by multipliers on $\cI_n$. 
A representation of $\cI_n$ is non-degenerate if and only if 
its restriction to $\cA_n$ is non-degenerate because 
$\cA_n \cI_n = \cI_n$. 

If a representation $(\rho,\cK)$ of $\cA$ is non-degenerate on 
all these ideals, then it is non-degenerate on every $\cA_n$, 
hence defines a continuous unitary representation of~$G$. 
\end{remark}

\begin{remark} Theorem~\ref{thm:dirlim} implies in particular 
the validity of the disintegration arguments 
in \cite[Thm.~3.6]{Ol78} and 
\cite[Prop.~2,4]{Pi90}. In \cite[Lemma~2.6]{Ol84} one also finds a very 
brief argument concerning the disintegration of ``holomorphic'' 
representations, namely that all the constituents are again 
``holomorphic''. We think that this is not obvious and requires additional 
arguments. 
\end{remark}

\end{document}